\documentclass[reqno,11pt]{article}
\usepackage{a4wide}
\usepackage{color,eucal,enumerate,mathrsfs, amsthm}
\usepackage[normalem]{ulem}
\usepackage{amsmath,amssymb,epsfig,bbm}
\numberwithin{equation}{section}

\newcommand{\fr}{\penalty-20\null\hfill$\blacksquare$}                      

\usepackage[pdfborder={0 0 0}]{hyperref}  

\usepackage[latin1]{inputenc}
\usepackage[english]{babel}

\newtheorem{theorem}{Theorem}[section]

\newtheorem{lemma}[theorem]{Lemma}
\newtheorem{proposition}[theorem]{Proposition}
\newtheorem{definition}[theorem]{Definition}
\newtheorem{assumption}[theorem]{Assumption}
\newtheorem{remark}[theorem]{Remark}

\newcommand{\geo}{\rm Geo}

\newcommand{\esup}[1]{{\mathrm{ess~sup}}{#1}}
\newcommand{\BL}{{\sf BL}}
\newcommand{\restr}[1]{\lower3pt\hbox{$|_{#1}$}}

\newcommand{\mm}{\mathfrak m}  
\newcommand{\ms}{(X,\d,\mm)}

\newcommand{\rcd}{{\rm RCD}(K, \infty)}

\newcommand{\eps}{\varepsilon}
\newcommand{\nchi}{{\raise.3ex\hbox{$\chi$}}}

\newcommand{\limi}{\varliminf}
\newcommand{\lims}{\varlimsup}
\newcommand{\lmt}[2]{\mathop{\lim}_{{#1} \rightarrow {#2}} }
\newcommand{\lmts}[2]{\mathop{\overline{\lim}}_{{#1} \rightarrow {#2}} }

\newcommand{\supp}{\mathop{\rm supp}\nolimits}   
\newcommand{\Lip}{\mathop{\rm Lip}\nolimits} 
\newcommand{\lip}[1]{{\rm{lip}}({#1})}
\newcommand{\loclip}[2]{{\rm{lip}}_{#2}({#1})}         
\renewcommand{\d}{{\mathrm d}}
\newcommand{\dt}{{\d t}}

\newcommand{\D}{{\mathrm D}}
\newcommand{\N}{\mathbb{N}}

\newcommand{\R}{\mathbb{R}}
\newcommand{\Z}{\mathbb{Z}}

\title{ Sobolev spaces on warped products}

\begin{document}
\author{ Nicola Gigli\thanks{SISSA, ngigli@sissa.it
}\and Bang-Xian Han \thanks
{ Universit\"{a}t Bonn, han@iam.uni-bonn.de}
}

\date{\today}
\maketitle

\begin{abstract}
We study  the structure of Sobolev spaces on the cartesian/warped products of a given metric measure space and an interval.

Our main results are:
\begin{itemize}
\item[-] the characterization of the Sobolev spaces in such products 
\item[-] the proof that, under natural assumptions, the warped products possess the Sobolev-to-Lipschitz property, which is key for geometric applications.
\end{itemize} 
The results of this paper have been needed in the recent proof of the `volume-cone-to-metric-cone' property of ${\sf RCD}$ spaces obtained by the first author and De Philippis.
 \end{abstract}

\textbf{Keywords}: warped product, Sobolev space, metric measure space.\\

\textbf{MSC2010}:  53C23, 46E35
\tableofcontents

\section{Introduction}
There is a well established definition of the  space $W^{1,2}(X,\d,\mm)$ of real valued Sobolev functions defined on a metric measure space $(X,\d,\mm)$ (\cite{Cheeger00}, \cite{Shanmugalingam00},  \cite{AmbrosioGigliSavare11}). A function $f\in W^{1,2}(X,\d,\mm)$ comes with a function $|\D f|_X\in L^2(X,\mm)$, called minimal weak upper gradient, playing the role of what the modulus of the distributional differential is in the smooth setting.

In this paper we are interested in the structure of the Sobolev spaces and the corresponding minimal weak upper gradients under some basic geometric constructions. The basic problem is the following. Let $(X,\d_X,\mm_X)$ and $(Y,\d_Y,\mm_Y)$ be two metric measure spaces and consider the space $X\times Y$ endowed with the product measure $\mm_c:=\mm_X\times\mm_Y$ and the product distance $\d_c$ defined as
\[
\d_c^2\big((x_1,y_1),(x_2,y_2)\big):=\d_X^2(x_1,x_2)+\d_Y^2(y_1,y_2),\qquad\forall x_1,x_2\in X,\ y_1,y_2\in Y.
\]
Then one asks what is the relation between Sobolev functions on $X\times Y$ and those on $X,Y$. Guided by the Euclidean case, one might conjecture that $f\in W^{1,2}(X\times Y)$ if and only if for $\mm_X$-a.e.\ $x$ the function $y\mapsto f(x,y)$ is in $W^{1,2}(Y)$, for $\mm_Y$-a.e.\ $y$ the function $x\mapsto f(x,y)$ is in $W^{1,2}(X)$ and the quantity
\[
\sqrt{|\D f(\cdot,y)|_X^2(x)+|\D f(x,\cdot)|_Y^2(y)}
\]
is in $L^2(X\times Y,\mm_c)$. Then one expects the above quantity to coincide with $|\D f|_{X\times Y}$.

Curiously, this kind of problem has not been studied until recently and, despite the innocent-looking statement, the full answer is not yet known. 

The first  result in this direction has been obtained in \cite{AmbrosioGigliSavare11-2}, where it has been proved that the conjecture is true under the very restrictive assumption that the spaces considered satisfy the, there introduced, $\rcd$ condition for some $K\in\R$. Such restriction was necessary to use some regularization property of the heat flow.

The curvature condition has been dropped in the more recent paper \cite{Ambrosio-Pinamonti-Speight15}. There the authors prove that the above conjecture holds provided either both the base spaces are doubling and support a weak local 1-2 Poincar\'e inequality, or on both the spaces the integral of the local Lipschitz constant squared is a quadratic form on the space of Lipschitz functions.

\bigskip

Our first contribution to the topic is the proof that the above conjecture is always true, provided one of the two spaces is $\R$ or a closed subinterval of $\R$. Our strategy is new and also allows to cover the case of warped product of a space and a closed interval, thus permitting to consider basic geometric constructions like that of cone and spherical suspension of a given space, which are in fact our main concern.

\medskip

The  second main result  of the paper concerns  the Sobolev-to-Lipschitz property (see Section \ref{se:stl} for the definition) of a warped product. Such notion, introduced in \cite{Gigli13} (see also \cite{Gigli13over}), is key to deduce precise metric information from the study of Sobolev functions and it is therefore important  to ask whether warped products have this property.  We will show that this is the case under very general assumptions.

\bigskip

Let us explain the role of this manuscript in relation with the literature on ${\rm RCD}$ spaces. This project has been motivated by the study of the `volume-cone-to-metric-cone' property of ${\rm RCD}$ spaces, obtained in \cite{DPG16}, where the results of this manuscript have been used in a crucial way. The main result in \cite{DPG16} is the proof that, under appropriate assumptions on the volume of concentric balls, the ball $B_R(\bar x)$ on a ${\rm RCD}(0,N)$ space is isomorphic to the cone, call it $C$, built over the sphere $S_R(\bar x)$ equipped with the intrinsic distance induced by the embedding of the sphere on the space and the appropriate measure. The hard part is the proof that $C$ and $B_R(\bar x)$ are isometric and much like in the proof of the non-smooth splitting this is achieved by:
\begin{itemize}
\item[a)] Showing that there is a bijection between the spaces which induces, by right composition, an isometry of the Sobolev spaces $W^{1,2}$
\item[b)] Proving that $C$ has, at least locally, the Sobolev-to-Lipschitz property (for $B_R(\bar x)$ this is already known from \cite{AmbrosioGigliSavare11-2}), so that from the previous point one can conclude. 
\end{itemize}
In order to tackle point $(a)$ one needs first to know the structure of Sobolev functions on the cone $C$ and in particular their relation with those on the sphere $S_R(\bar x)$, because it is this latter space that is directly linked to  $B_R(\bar x)$. Clarifying this relation is our first main result. Obtaining point $(b)$ under a sufficiently general set of assumptions is our second.

Let us underline that \emph{a posteriori}, once the isomorphism between $C$ and $B_R(\bar x)$ has been built, one obtains - via Ketterer's results in \cite{Ketterer13} - that the sphere $S_R(\bar x)$ is a ${\rm RCD}(N-2,N-1)$ space. However, \emph{a priori} very little is known about its structure so that one does not know whether the previous results in \cite{AmbrosioGigliSavare11-2},  \cite{Ambrosio-Pinamonti-Speight15} can be adapted to cover this situation: this is why it is necessary to work with minimal hypothesis on our base space $X$.

\section{Preliminaries}

\subsection{Metric measure spaces}
Let $(X,\d)$ be a  complete metric  space. By a curve  $\gamma$ we shall typically denote  a continuous map $\gamma: [0,1] \mapsto X$, although sometimes curves defined on different intervals will be considered. The space of  curves on $[0,1]$ with values in $X$ is denoted by $C([0,1],X)$.  The space $C([0,1],X)$ equipped with the uniform distance  is a  complete metric space.

We define the length of $\gamma$ by
\[
l[\gamma]:=\sup_\tau \mathop{\sum}_{i=1}^{n} \d(\gamma(t_{i-1}),\gamma(t_i))
\]
where $\tau:=\{0=t_0, t_1, ..., t_n=1\}$ is a partition of  $[0,1]$. The supremum here can be changed to `$\lim$' and the limit  is taken with respect to the refinement ordering of partitions.

The space $(X, \d)$ is said to be a length space if for any $x,y \in X$ we have
\[
\d(x,y) = \inf_\gamma l[\gamma]
\]
where the infimum is taken among all $\gamma \in C([0,1],X)$ which connect $x$ and $y$.

If the infimum is always a minimum, then the space is called geodesic space and we call the minimizers pre-geodesics. A geodesic from $x$ to $y$
is any pre-geodesic which is parametrized by constant speed. Equivalently, a geodesic from $x$ to $y$ is a curve $\gamma$ such that:
\[
\d(\gamma_s,\gamma_t) = |s - t|\d(\gamma_0,\gamma_1),~~~~~~~\forall t, s \in [0,1],~~~\gamma_0 = x,\gamma_1 = y.
\]
The space of all geodesics on X will be denoted by $\geo(X)$. It is a closed subset of $C([0,1],X)$.

Given $p\in [1,+\infty]$ and a curve $\gamma$, we say that $\gamma$ belongs to $AC^p([0,1],X)$ if
\[
\d(\gamma_s,\gamma_t) \leq \int_s^t G(r) \,\d r,~~~~\forall ~t,s \in [0,1], ~s<t
\]
for some $G \in L^p([0,1])$. In particular, the case $p=1$ corresponds to  absolutely continuous curves, whose class is denoted by $AC([0,1], X)$.
It is known (see for instance Theorem 1.1.2 of \cite{AmbrosioGigliSavare08}) that for $\gamma \in  AC([0,1], X)$, there exists an a.e. minimal function $G$ satisfying this inequality, called the metric derivative
which can be computed for a.e. $t\in [0,1]$ as
\[
|\dot{\gamma}_t|:=\lmt{h}{0}\frac{\d(\gamma_{t+h},\gamma_t)}{|h|}.
\]

It is known that (see for example \cite{BBI01}) the length of a curve $\gamma \in  AC([0,1], X)$ can be computed as
\[
l[\gamma]:=\int_0^1 |\dot{\gamma}_t|\,\dt.
\]
In particular, on a length space $X$ we have
\[
\d(x,y) = \inf_\gamma \int_0^1 |\dot{\gamma}_t|\,\dt
\]
where the infimum is taken among all $\gamma \in AC([0,1],X)$ which connect $x$ and $y$.

Given $f : X \mapsto \mathbb{R}$, the local Lipschitz constant $\lip{ f}: X \mapsto [0, \infty]$ is defined as
\[
\lip{f}(x):= \mathop{\overline{\lim}}_{y\rightarrow x}\frac{|f(y) -f(x)|}{\d(x, y)}
\]
if $x$ is not isolated, $0$ otherwise, while the  (global) Lipschitz constant  is defined as
\[
\Lip(f):= \mathop{\sup}_{x \neq y} \frac{|f(y)-f(x)|}{\d(x,y)}.
\]
If $(X, \d)$ is a length space, we have $\Lip(f)=\mathop{\sup}_{x}  \lip{f}(x)$.

We are not only  interested in metric structure, but also in the interaction between metric and measure. For the metric measure space $\ms$, basic assumptions used in this paper are:

\begin{assumption}\label{assumption}
The metric measure space $\ms$ satisfies:
\begin{itemize}
\item $(X,\d)$ is a complete and separable length space,
\item $\mm$ is a non-negative Borel  measure with respect to $\d$ and finite on bounded sets,
\item $\supp{\mm}=X$.
\end{itemize}

\end{assumption}

Moreover, for brevity  we will not distinguish $X, (X,\d)$ or $\ms$ when no ambiguity  exists. For example, we write  $S^2(X)$ instead of $S^2(X,\d, \mm)$ (see the next section).

\subsection{Sobolev functions}

\begin{definition}[Test plan] Let $(X,\d,\mm)$ be a metric measure space and $\pi \in \mathcal{P}(C([0,1],X))$. We say  that $\pi$ has bounded compression provided
there exists $C >0$ such that
\[
(e_t)_\sharp \pi \leq C\mm,~~~\forall t \in [0,1].
\]

Then we say that $\pi$ is a test plan if it has bounded compression, is concentrated on $AC^2([0,1],X)$ and
\[
\int_0^1\int |\dot{\gamma}_t|^2 \,\d\pi(\gamma)\, \dt < +\infty.
\]
\end{definition}

The notion of Sobolev function is given by duality with that of test plan:
\begin{definition}[Sobolev class] Let $\ms$ be a metric measure space. A Borel function $f : X \rightarrow \mathbb{R}$ belongs to the
Sobolev class $S^2\ms$ (resp. $S^2_{loc}\ms$) provided there exists a non-negative function $G\in L^2(X,\mm)$ (resp. $L^2_{loc}(X,\mm)$) such that

\[
\int |f(\gamma_1)- f(\gamma_0)|\, \d\pi(\gamma) \leq \int \int_0^1 G(\gamma_s)|\dot{\gamma}_s|\, \d s\, \d\pi(\gamma),  ~~~\forall ~\text{test plan}~\pi.
\]

In this case, $G$ is called a 2-weak upper gradient of $f$, or simply  weak upper gradient.
\end{definition}
It is known, see e.g.\ \cite{AmbrosioGigliSavare11}, that there exists a minimal function $G$ in the $\mm$-a.e. sense among all the weak upper gradients of $f$. We denote such minimal function
by $|\D f|$ or $|\D f|_X$ to emphasize which  space we are considering and call it  minimal weak upper gradient. Notice that if $f$ is Lipschitz, then $|\D f|\leq \lip f$ $\mm$-a.e., because $\lip f$ is a weak upper gradient of $f$.

It is known that the locality holds for $|\D f|$, i.e. $|\D f|=|\D g|$ a.e. on the set $\{ f=g\}$, moreover $S^2_{loc}\ms$ is a vector space and the inequality
\begin{equation}
\label{eq:sumd}
|\D(\alpha f+\beta g)|\leq |\alpha||\D f|+|\beta||\D g|,\qquad\mm-a.e.,
\end{equation}
holds for every $f,g\in S^2_{loc}\ms$ and $\alpha,\beta\in \R$ and the space $S^2_{loc}\cap L^\infty_{loc}\ms$ is an algebra, with the inequality
\begin{equation}
\label{eq:leibn}
|\D(fg)|\leq |f||\D g|+|g||\D f|,\qquad\mm-a.e.,
\end{equation}
being valid for any $f,g\in S^2_{loc}\cap L^\infty_{loc}\ms$.

Another basic - and easy to check - property of minimal weak upper gradients that we shall frequently use is their semicontinuity in the following sense: if $(f_n)\subset S^2\ms$ is a sequence $\mm$-a.e.\ converging to some $f$ and such that $(|\D f_n|)$ is bounded in $L^2(X,\mm)$, then $f\in S^2\ms$ and 
\[
|\D f|\leq G,\qquad\mm-a.e.,
\]
for every $L^2$-weak limit $G$ of some subsequence of $(|\D f_n|)$ (see \cite{AmbrosioGigliSavare11}).

Then the Sobolev space $W^{1,2}\ms$ is defined as $W^{1,2}\ms:= S^2\ms  \cap L^2(X,\mm)$ and is  endowed with the norm
\[
\|f\|^2_{W^{1,2}\ms}:=\|f\|^2_{L^2(X,\mm)}+\||\D f|\|^2_{L^2(X,\mm)}.
\]
$W^{1,2}(X)$ is always a Banach space, but in general it is not an Hilbert space. Following \cite{Gigli12}, we say that $\ms$ is an infinitesimally Hilbertian space if $W^{1,2}(X)$ is an Hilbert space.

In \cite{AmbrosioGigliSavare11} (see also \cite{AmbrosioGigliSavare11-3})  the following result has been proved.

\begin{proposition}[Density in energy of Lipschitz functions]\label{prop-density}
Let $\ms$ be a metric measure space and $f\in W^{1,2}(X)$. Then there exists a sequence $(f_n)$ of Lipschitz functions $L^2$-converging to $f$ such that the sequence $(\lip{f_n})$ $L^2$-converges to $|\D f|$.
\end{proposition}

\subsection{Product spaces}
In this subsection we recall the basic concepts and results about the Cartesian product  and the warped product of two spaces.  Both metric and metric measure structures are considered.

Given two metric measure spaces $(X, \d_X, \mm_X)$ and $(Y, \d_Y, \mm_Y)$, we define their (Cartesian) product as:

\begin{definition}[Cartesian product]
We define the space $(Y \times X, \d_c, \mm_c)$ as the product space $Y\times X$ equipped with the distance $\d_c:=\d_Y \times \d_X$ and the measure $\mm_c:=\mm_Y \times \mm_X$. Here $\d_c=\d_Y \times \d_X$ means:
\[
\d_c ((y_1,x_1), (y_2,x_2))= \sqrt{\d_Y^2(y_1, y_2)+ \d_X^2(x_1, x_2)},
\]
 for any pairs $(y_1, x_1), (y_2, x_2) \in Y \times X$.
\end{definition}
We shall make use of the following simple result, established in \cite{Gigli13}, linking Sobolev functions on the base spaces with those in the product:
\begin{proposition}\label{prop:linkprod}
Let  $g\in L^2_{loc}(X)$ and define $f\in L^2_{loc}( Y\times X)$ as $f(y,x):=g(x)$.

Then $f\in S^2_{loc}(Y\times X)$ if and only if $g\in S^2_{loc}(X)$ and in this case the identity
\[
|\D f|_{X_c}(y,x)=|\D g|_X(x),
\]
holds for $\mm_c$-a.e.\ $(t,x)$.
\end{proposition}

For warped products the construction is slightly more complicated. The warped product metric is defined for $X,Y$ length spaces only and in order to introduce it we need first to discuss the corresponding notion of length:
\begin{definition}[Warped length of curves]\label{def-warped-1}
Let $(X,\d_X)$ and $(Y,\d_Y)$ be two length   spaces and $w_\d:Y\to[0,+\infty)$  a continuous function. Let  $\gamma=(\gamma^Y,\gamma^X)$ be a curve such that $\gamma^Y,\gamma^X$ are absolutely continuous.
Then the $w_\d$-length of $\gamma$ is defined as
\[
l_w[\gamma]:=\lim_\tau \mathop{\sum}_{i=1}^{n} \sqrt{\d^2_Y(\gamma^Y_{t_{i-1}},\gamma^Y_{t_i})+w_\d^2(\gamma^Y_{t_{i-1}})\d^2_X(\gamma^X_{t_{i-1}},\gamma^X_{t_i})},
\]
where $\tau:=\{0=t_0, t_1, ..., t_n=1\}$ is a partition of  $I=[0,1]$ and the limit is taken with respect to the refinement ordering of partitions.
\end{definition}
It is not hard to check that the limit exists and that the formula
\[
l_w[\gamma]=\int_0^1 \sqrt{|\dot{\gamma}^Y_t|^2+w^2_\d(\gamma^Y_t)|\dot{\gamma}^X_t|^2} \,\d t
\]
holds.

Then we can define the metric $\d_w$ using this length structure:
\begin{definition}[Warped product of metric spaces]\label{def-warped-2} Let $(X,\d_X)$ and $(Y,\d_Y)$ be two length   spaces and $w_\d:Y\to[0,\infty)$  a continuous function. Define a pseudo-metric $\d_w$ on the space $Y\times X$ by
\[
\d_w(p,q):=\inf\{l_w[\gamma]:\gamma=(\gamma^Y,\gamma^X) ~\text{with $\gamma^Y,\gamma^X$  absolutely continuous and }\gamma_0=p,\ \gamma_1=q \},
\]
for any  $p,q \in Y\times X$.
\end{definition}
The pseudo metric $\d_w$ induces an equivalence relation on $Y\times X$ by $(y,x)\sim (y',x')$ iff $\d_w((y,x),(y',x'))=0$ and then a metric on the quotient. With a common slight abuse of notation we shall denote the completion of such quotient by $(Y\times_wX,\d_w)$. It is clear that if $X,Y$ are separable, then so is $Y\times_w X$. Let us denote by $\pi:Y\times X\to Y\times_wX$ the quotient map, then we can give the following definition:
\begin{definition}[Warped product of metric measure spaces]\label{def-warped-3} Let $(X,\d_X,\mm_X)$ and $(Y,\d_Y,\mm_Y)$ be two complete separable and length metric spaces equipped with non-negative Radon measures. Assume also that $\mm_X$ is a finite measure and let $w_\d,w_\mm:Y\to[0,+\infty)$ be continuous functions. 

Then the warped product $(Y\times_wX,\d_w)$ is defined as above and the Radon measure $\mm_w$ is defined as
\[
\mm_w:=\pi_*\big((w_\mm\mm_Y)\times\mm_X\big)
\]
\end{definition}
The assumption that $\mm_X$ is finite is needed to ensure that $\mm_w$, which is  always a Borel measure, is actually Radon. Indeed, observe that the trivial inequality
\[
\d_w\big((y,x),(y',x')\big)\geq \d_Y(y,y')
\]
grants that the projection map $\pi^Y:Y\times X\to Y$ passes to the quotient and induces a 1-Lipschitz map, still denoted by $\pi^Y$, from $Y\times_wX$ to $Y$. Then for $p\in Y\times_wX$ we can find a neighbourhood $U$ of $\pi^Y(p)$ in $Y$ such that $\mm_Y(U)<\infty$. It is then clear that $\pi(U\times X)$ is a neighbourhood of $p$ in $Y\times_wX$ of finite mass, thus proving the claim. If $\mm_X$ is not finite, it is still true that $\mm_w$ is a Radon measure, provided $w_\mm$ is never 0, but in applications to geometry it is often the case  that $w_\mm$ is 0 in at least one point, so that we shall always assume that $\mm_X$ is finite, even if all our results only require $\mm_w$ to be Radon.

Finally a word on notation. With a slight abuse, we shall denote the typical element of $Y\times_wX$ by $(y,x)$. This is not really harmful since the complement of $\pi(Y\times X)$ in $Y\times_wX$ is $\mm_w$-negligible and in writing a function on $Y\times_wX$ as a function on $Y\times X$ it will be implicitly understood that such function passes to the quotient.

\section{The results}
\subsection{Cartesian product}
Throughout this section $(X,\d,\mm)$ is a fixed complete, separable and length space and $I\subset\R$ a closed, possibly unbounded, interval. We are interested in studying the Cartesian product $(X_c,\d_c,\mm_c)$ of $I$, endowed with its Euclidean structure, and $(X,\d,\mm)$.

Given a function $f: X_c\to\R$ and $x\in X$ we denote by $f^{(x)}:I\to\R$ the function given by $f^{(x)}(t):=f(t,x)$. Similarly, for $t\in I$ we denote by $f^{(t)}:X\to \R$ the function given by $f^{(t)}(x):=f(t,x)$.

We  start introducing the Beppo Levi space  $\BL( X_c)$:
\begin{definition}[The space $\BL( X_c)$]
The space $\BL( X_c)\subset L^2( X_c,\mm_c )$ is the space of functions $f\in L^2( X_c,\mm_c)$ such that
\begin{itemize}
\item[i)] $f^{(x)}\in W^{1,2}(I)$ for $\mm$-a.e. $x$,
\item[ii)]$f^{(t)}\in W^{1,2}(X)$ for $\mathcal L^1$-a.e. $t$
\item[iii)] the function
\[
|\D f|_c(t,x):=\sqrt{|\D f^{(t)}|_X^2(x)+|\D f^{(x)}|_I^2(t)},
\]
belongs to $L^2( X_c,\mm_c)$.
\end{itemize}
On $\BL(X_c)$ we put the norm
\[
\|f\|_{\BL(X_c)}^2:=\| f\|_{L^2(X_c)}^2+\||\D f|_c\|^2_{L^2(X_c)}.
\]
The space $\BL_{loc}( X_c)$ is the subset of $L^2_{loc}( X_c,\mm_c)$ of functions which are locally equal to some function in $\BL( X_c)$.
\end{definition}

The main result of this section is the identification of the spaces $W^{1,2}(X_c)$ and $\BL(X_c)$ and of their corresponding weak gradients $|\D f|_{X_c}$ and $|\D f|_c$. One inclusion has been proved in \cite{AmbrosioGigliSavare11-2}, notice that although in \cite{AmbrosioGigliSavare11-2} a lower Ricci curvature bound is often present, the following result is stated for arbitrary $(X,\d,\mm)$ as above:
\begin{proposition} [Proposition 6.18 of \cite{AmbrosioGigliSavare11-2}]\label{ineq-car-1}
We have $ W^{1,2}( X_c)\subset\BL( X_c)$ and
\begin{equation}
\label{eq:cleqw}
\int_{ X_c}|\D f|_c^2\,\d\mm_c\leq \int_{  X_c} |\D f|^2_{ X_c}\,\d\mm_c,\qquad\forall f\in W^{1,2}( X_c).
\end{equation}
\end{proposition}
The key to proving the other inclusion is in the following purely metric lemma:
\begin{lemma}\label{lemma-leq}
Let  $f:X_c\to\R$ be of the form $f(t,x)=g_1(x)+ h(t)g_2(x)$ for Lipschitz functions  $g_1, g_2,h$. Then
\[
\lip {f}^2(t,x)\leq\loclip{f^{(t)}}{X}^2(x)+\loclip{f^{(x)}}{I}^2(t)
\]
for every $(t,x)\in  X_c$.
\end{lemma}
\begin{proof}
Let $(t,x),(s,y)\in   X_c$, and notice that
\[
\begin{split}
|f(s,y)-f(t,x)|&=|g_1(y)+h(s)g_2(y)-g_1(x)-h(t)g_2(x)|\\
&\leq\frac{|h(s)-h(t)||g_2(y)|}{|s-t|}|s-t|+\frac{|g_1(y)-g_1(x)+h(t)(g_2(y)-g_2(x))|}{\d(x,y)}\d(x,y).
\end{split}
\]
Hence from the  Cauchy-Schwarz inequality we obtain, after a division by $\d_c\big((s,y),(t,x)\big)$:
\[
\frac{|f(s,y)-f(t,x)|}{\d_c\big((s,y),(t,x)\big)}\leq  \sqrt{\frac{|h(s)-h(t)|^2|g_2(y)|^2}{|s-t|^2}+\frac{|g_1(y)-g_1(x)+h(t)(g_2(y)-g_2(x))|^2}{\d^2(x,y)}}.
\]
Letting $(s,y)\to(t,x)$ and using the continuity of $g_2$ we conclude.
\end{proof}
In this last lemma, the fact that $I$ was an interval played no role;  to realize the importance of this restriction and streamline the argument it is useful to introduce the following classes of functions:
\begin{definition}[The classes $\mathcal A$ and $\tilde{\mathcal A}$]
We define the space of functions $\mathcal A\subset \BL_{loc}( X_c)$ as
\[
\begin{split}
\mathcal{A} := \Big\{g_1(x)+h(t)g_2(x)\in  \BL_{loc}( X_c)\ :\   g_1, g_2 \in W^{1,2}{(X)}, ~~h: I\to\R\ \text{ is Lipschitz } \Big\},
\end{split}
\]
and the space $\tilde{\mathcal{A}}\subset \BL_{loc}( X_c)$ as the set of functions $f\in \BL_{loc}( X_c) $ which are locally equal to some function in $\mathcal{A}$.
\end{definition}
Notice that Proposition \ref{prop:linkprod} and the calculus rules \eqref{eq:sumd}, \eqref{eq:leibn} ensure that
\begin{equation}
\label{eq:ainclw}
\tilde{\mathcal A}\subset S^2_{loc}(X_c).
\end{equation}

The interest of functions in $\tilde{\mathcal A}$ is due to the next two results:

\begin{proposition}\label{eq-car-1}
Let  $f \in \tilde{ \mathcal{A}}$. Then
\[
|\D f|_{  X_c}=|\D f|_c \qquad\mm_c-a.e..
\]
\end{proposition}
\begin{proof} Notice that by \eqref{eq:ainclw}  the statement makes sense. Moreover, due to the local nature of the statement we can assume that  $f(t,x)=g_1(x)+h(t)g_2(x)\in \mathcal A$ with $h$ having compact support. With this assumption we have that $f\in W^{1,2}(X_c)$ so that keeping in mind  Proposition \ref{ineq-car-1}, to conclude it is sufficient to prove that
\begin{equation}
\label{eq:toprove}
|\D f|_{X_c}^2(t,x)\leq|\D f^{(x)}|^2_I(t)+|\D f^{(t)}|^2_X(x),\qquad\mm_c-a.e.\ (t,x).
\end{equation}
To this aim, it is in turn sufficient to show that for any $[a,b)\subset I$ and any Borel set $E\subset X$ we have
\begin{equation}
\label{eq:toprove2}
\int_{\tilde E}|\D f|_{ X_c}^2(t,x)\,\d t\,\d\mm(x)\leq\int_{\tilde E}|\D f^{(x)}|^2_I(t)+|\D f^{(t)}|^2_X(x)\,\d t\,\d\mm(x)
\end{equation}
with $\tilde E:=[a,b)\times E$. Indeed if this holds, taking into account that open sets in $ X_c$ can always be written as disjoint countable union of sets of the form $[a,b)\times E$, we deduce that \eqref{eq:toprove2} holds with $\tilde E$ generic open set in $ X_c$, so that using the fact that the integrand are in $L^1( X_c,\mm_c )$ by exterior approximation we get that \eqref{eq:toprove2} holds for arbitrary Borel sets $\tilde E\subset   X_c$ and thus \eqref{eq:toprove} and the conclusion.

Thus fix $E\subset X$ Borel, let $\tilde E:=[a,b)\times E $ and up to a simple scaling argument assume also that $[a,b)=[0,1)$.

For  $k,i\in\N$, $k>0$, we define $f_{k,i}\in W^{1,2}(X) $ as  $f_{k,i}(x):= g_1(x)+h(\tfrac{i}k) g_2(x)$ and $f_k\in \BL(X_c)$ as 
\[
f_k(t,x):=(k t- i)f_{k,i+1}(x)+( i+1-kt)f_{k,i}(x),\qquad\text{ for }t\in \big[\tfrac ik,\tfrac{i+1}k\big].
\]
Notice that $f_k\to f$ in $L^2(X_c,\mm_c)$. By Proposition \ref{prop-density}, for each $(k,i)$ we can find a sequence of Lipschitz functions $f_{k,i,n}\in\Lip{(X)}$ converging to $f_{k,i}$ in $L^2(X,\mm)$ such that $\lim_{n \rightarrow \infty} \lip {f_{k,i,n}}  = |\D f_{k,i}|_X $ in $L^2(X,\mm)$. It will also be convenient to define $g_{k,n},g_{k}\in L^2(X_c,\mm_c)$ as
\[
g_{k,n}(\cdot,t):=k(f_{k,i+1,n}-f_{k,i,n})\qquad\qquad g_{k}(\cdot,t):=k(f_{k,i+1}-f_{k,i})\qquad\qquad\text{ for }t\in (\tfrac ik,\tfrac{i+1}k)
\]
and to notice that
\begin{equation}
\label{eq:limg}
\lim_{n\to\infty} g_{k,n}=g_k\qquad\qquad\lim_{k\to\infty} g_{k}=|\D f^{(\cdot)}|_I(\cdot),
\end{equation}
both limits being in $L^2(X_c,\mm)$: here the first limit is obvious by construction, while the second follows from the identity $g_k(t,x)=k(h(\frac{i+1}k)-h(\frac ik))g_2(x)$ valid  for $t\in (\tfrac ik,\tfrac{i+1}k)$.

Now define $F_{k,n}\in\Lip(  X_c)$ as
\[
F_{k,n}(t,x):=(k t- i)f_{k,i+1,n}(x)+( i+1-kt)f_{k,i,n}(x),\qquad\text{ for }t\in \big[\tfrac ik,\tfrac{i+1}k\big]
\]
and notice that construction we have $F_{k,n}\in\tilde{\mathcal A}$, so that  Lemma \ref{lemma-leq}  gives
\[
|\lip {F_{k,n}}|^2 \leq   |\loclip{ F_{k,n}}{X}|^2 +|\loclip{F_{k,n}}{I}|^2,\qquad\mathcal L^1\times\mm-a.e.,
\]
for every $k,n$ and since we  also have   $\lim_{k} \lim_{n} F_{k,n}=f$ in $L^2( X_c, \mm_c)$, such bound together with  the  lower semicontinuity of minimal weak upper gradients gives that
\begin{equation}
\label{eq:step1}
\int_{\tilde E} |\D f|_{X_c}^2\,\d\mm_c \leq \lims_{k\to\infty}\lims_{n\to\infty}\int_{\tilde E}\loclip{ F_{k,n}}{X}^2\,\d\mm_c  +\lims_{k\to\infty}\lims_{n\to\infty}\int_{\tilde E}\loclip{F_{k,n}}{I}^2\,\d\mm_c
\end{equation}
so that to get \eqref{eq:toprove2} - and thus the conclusion - it is sufficient to prove that
\begin{equation}
\label{eq:pieces}
\begin{split}
 \lims_{k\to\infty}\lims_{n\to\infty}\int_{\tilde E}\loclip{ F_{k,n}}{X}^2\,\d\mm_c &\leq \int_{\tilde E}|\D f^{(t)}|^2_X(x)\,\d \mm_c(t,x),\\
 \lims_{k\to\infty}\lims_{n\to\infty}\int_{\tilde E}\loclip{F_{k,n}}{I}^2\,\d\mm_c&\leq \int_{\tilde E}|\D f^{(x)}|^2_I(t)\,\d \mm_c(t,x).
\end{split}
\end{equation}
To this aim, start observing that 
the continuity of $h$ grants that $\R\ni t\mapsto f^{(t)}\in W^{1,2}(X)$ is continuous so that also the map $I\ni t\mapsto\int _E|\D f^{(t)}|_X^2\,\d\mm$ is continuous. In particular, its integral on $[0,1]$  coincides with the limit of the Riemann sums:
\begin{equation}
\label{eq:riemann}
\begin{split}
\int_{\tilde E}|\D f^{(t)}|^2_X(x)\,\d\mm_c(t,x)&=\lim_{k\to\infty}\frac1k\sum_{i=0}^k\int_E|\D f_{k,i} |_X^2\,\d\mm=\lim_{k\to\infty}\lim_{n\to\infty}\frac1k\sum_{i=0}^k\int_E\loclip  {f_{k,i,n}}{X}^2 \,\d\mm.
\end{split}
\end{equation}
From the very definition of $F_{k,n}$ we get that
\[
\begin{split}
\loclip{F^{(t)}_{k,n}}{X}^2&\leq \big((k t- i)\loclip{f_{k,i+1,n}}X+(i+1-kt)\loclip{f_{k,i,n}}X\big)^2\\
&\leq (k t- i)\loclip{f_{k,i+1,n}}X^2+( i+1-kt)\loclip{f_{k,i,n}}X^2,
\end{split}
\]
on $X$ for every $t\in  [\tfrac ik,\tfrac{i+1}k]$, and thus
\[
\begin{split}
\int_{\tilde E} \loclip{F^{(t)}_{k,n}}{X}^2(x)\,\d\mm_c(t,x) &\leq \int_X\frac1k\sum_{i=0}^k  \loclip  {f_{k,i,n}}{X}^2-\frac1{2}\Big(\loclip{ f_{k,0,n}}{X}^2+\loclip{f_{k,k,n}}{X}^2\Big)\,\d\mm\\
&\leq\int_X \frac1k\sum_{i=0}^k  \loclip  {f_{k,i,n}}{X}^2\,\d\mm.
\end{split}
\]
This inequality and \eqref{eq:riemann} give the first in \eqref{eq:pieces}. The second is a direct consequence of the fact that $\loclip{F_{k,n}}{I}=g_{k,n}$ and the limiting properties \eqref{eq:limg}.\end{proof}

\begin{proposition}[Density in energy]\label{prop-density-2}
For any function $f\in \BL( X_c)$  there exists a sequence  $(f_n)\subset  \BL(  X_c) \cap \tilde{\mathcal{A}}$ converging to $f$ in $L^2(   X_c,\mm_c)$ such that $|\D f_n|_c\to|\D f|_c$ in $L^2(  X_c,\mm_c)$  as $n\to\infty$.
\end{proposition}
\begin{proof} We shall give the proof for the case $I=\R$, the argument for arbitrary $I$ being similar.

With a standard cut-off, truncation and diagonalization argument we can, and will, assume that the given $f\in \BL(  X_c)$ is bounded and with bounded support. Then  for any $n \in \mathbb{N}$ and $i\in \mathbb{Z}$ we define
\[
g_{i,n}(x):=n\int_\frac{i}n^\frac{(i+1)}n f(x,s)\,ds,
\]
and
\[
h_{i,n}(t):=\chi_n\big(t-\frac{i}n\big),
\]
where $\chi_n:\mathbb{R} \mapsto \mathbb{R}$ is given by:
\begin{equation}
\chi_n(t):=
\left\{
\begin{array}{llll}
0,~~~~~~~~~~~&\text{if}&~t<-\frac1n,\\
nt+1,&\text{if}&~-\frac1n\leq t<0,\\
1-nt,&\text{if}&~0\leq t<\frac1n,\\
0,&\text{if}&~\frac1n<t.
\end{array}\right.
\end{equation}
Then we define the sequence $(f_n)$ as:
\[
f_n(t,x) := \mathop{\sum}_{ i \in \mathbb{Z}} h_{i,n}(t) g_{i,n}(x),
\]
the sum being well defined because $g_{i,n}$ is not zero only for a finite number of $i$'s and it is immediate to check that  $f_n\in \tilde{\mathcal A}$.

We claim that $f_n\to f$ in $L^{2}( X_c,\mm_c)$ as $n\to\infty$. Integrating the inequality
\[
\begin{split}
\big(f_n(t,x)\big)^2&=\left(\sum_{i\in \Z}h_{i,n}(t)g_{i,n}(x)\right)^2\leq\sum_{i\in \Z}h_{i,n}(t)\big(g_{i,n}(x)\big)^2\leq \sum_{i\in \Z}h_{i,n}(t)n\int_{i/n}^{(i+1)/n}f^2(s,x)\,\d s,
\end{split}
\]
on $x$ and $t$ we obtain $\|f_n\|_{L^2( X_c)}\leq \|f\|_{L^2( X_c)}$, for every  $ n\in\N$. This means that the linear operator $T_n$ from $L^2( X_c,\mm_c)$ into itself assigning $f_n$ to $f$ is 1-Lipschitz for every $n\in\N$. Since obviously $f_n\to f$ in $L^2( X_c,\mm_c)$ if $f$ is Lipschitz with bounded support, the uniform continuity of the $T_n$'s grant that $f_n\to f$ in $L^2( X_c,\mm_c)$ for every $f\in L^2(X_c,\mm_c)$.

Now, taking into account the $L^2$-lower semicontinuity of the $\BL$-norm, to conclude it is sufficient to show that for every $n\in\N$ we have
\begin{equation}
\label{eq:perfinirew12}
\begin{split}
\int_{X_c} |\D  f^{(t)}_n|^2_{X}(x)\,\d\mm_c(t,x)&\leq\int_{\R\times X} |\D  f^{(t)}|^2_{X}(x)\,\d\mm_c(t,x),\\
\int_{\tilde   X} |\D  f^{(x)}_n|^2_{\R}(t)\,\d\mm_c(t,x)&\leq\int_{\R\times X} |\D  f^{(x)}|^2_{\R}(t)\,\d \mm_c(t,x).
\end{split}
\end{equation}
Start noticing that the definition of the functions $g_{i,n}$, and Jensen's inequality applied to the convex and lower semicontinuous function on $L^2(X)$ which sends $g$ to $\int |\D g|_X^2\,\d\mm$ (intended to be $+\infty$ if $g\notin W^{1,2(X)}$) we see that $g_{i,n}\in W^{1,2}(X)$ with 
\begin{equation}
\label{eq:venti}
\int_{X}|\D g_{i,n}|^2_{X}\,\d\mm\leq n\int_{X}\int_{i/n}^{(i+1)/n}|\D f^{(t)}|^2_{X'}\,\d t\,\d\mm.
\end{equation}
Then from the trivial identity
\[
f_n^{(t)}=(1+i-nt)g_{i,n}+(nt-i)g_{i+1,n},
\]
valid for every $n$ and a.e.\ $t\in[\tfrac in,\tfrac{i+1}n]$ we know that $f^{(t)}_n\in W^{1,2}(X)$ and
\[
\begin{split}
|\D f_n^{(t)} |_{X}^2&\leq \big((1+i-nt)|\D g_{i,n}|_{X}+(nt-i)|\D g_{i+1,n}|_{X}\big)^2\\
&\leq(1+i-nt)|\D g_{i,n}|_{X}^2+(nt-i)|\D g_{i+1,n}|_{X}^2,
\end{split}
\]
 for every $n$ and a.e.\ $t\in[\tfrac in,\tfrac{i+1}n]$. This yields the bound
\begin{equation}
\label{eq:gradhor}
\begin{split}
\int_{  X_c}|\D f_n^{(t)} |_{X}^2(x)\,\d\mm_c(t,x)&\leq \frac1n\sum_{i\in\Z}\int_{X}|\D g_{i,n}|_{X}^2(x)\,\d\mm( x)\\
\text{by \eqref{eq:venti}\qquad\qquad}&\leq \sum_{i\in\Z}\int_{X}\int_{i/n}^{(i+1)/n}|\D f^{(t)}|_{X}^2(x)\,\d t\,\d\mm(x)\\
&=\int_{ X_c}|\D f^{(t)}|^2_{X}(x)\,\d\mm_c (t,x),
\end{split}
\end{equation}
which is the first in \eqref{eq:perfinirew12}.

Similarly, for $\mm$-a.e.\ $x\in X$ the function $f_n^{(x)}:\R\to\R$ is $\mathcal L^1$-a.e.\ well defined and given by
\[
f_n^{(x)}(t)=(1+i-nt)g_{i,n}(x)+(nt-i)g_{i+1,n}(x),\qquad\mathcal L^1-a.e.\    t\in\big [\tfrac in,\tfrac{i+1}n\big].
\]
Arguing as before we get that $f_n^{(x)}\in W^{1,2}(\R)$ for $\mm$-a.e.\ $x$ and
\[
\begin{split}
\int_{i/n}^{(i+1)/n}|\D f_n^{(x)}|^2_{\R}(t)\,\d t&=\int_{i/n}^{(i+1)/n}n^2\big( g_{i+1,n}(x)-g_{i,n}(x)\big)^2\,\d t\\
&=n\big( g_{i+1,n}(x)-g_{i,n}(x)\big)^2\\
&=n^3\left(\int_{(i+1)/n}^{(i+2)/n}f(t,x)\,\d t-\int_{i/n}^{(i+1)/n}f(t,x)\,\d t\right)^2\\
&=n^3\left(\int_{i/n}^{(i+1)/n}f^{(x)}(t+1/n)-f^{(x)}(t)\,\d t\right)^2\\
&\leq n^3\left(\int_{i/n}^{(i+1)/n}\int_t^{t+1/n}|\D f^{(x)}|_\R(s)\,\d s\,\d t\right)^2\\
&\leq n\int_{i/n}^{(i+1)/n}\int_t^{t+1/n}|\D f^{(x)}|_\R^2(s)\,\d s\,\d t,
\end{split}
\]
which  yields
\[
\int_{X_c}|\D f_n^{(x)}|^2_{\R}(t)\,\d\mm_c(t,x)  \leq \int_{X_c}|\D f^{(x)}|^2_{\R}(t)\,\d\mm_c(t,x),
\]
which is the second in \eqref{eq:perfinirew12} and the conclusion.
\end{proof}

We now have all the tools to prove the main result of this section:

\begin{theorem}\label{theorem-car} The sets   $W^{1,2}(X_c )$ and $\BL(  X_c)$ coincide and for every $f\in W^{1,2}(  X_c)=\BL( X_c)$ the identity

\[
|\D f|_{ X_c}=|\D f|_c\qquad  \mm_c-a.e.,
\]
holds.
\end{theorem}
\begin{proof}
Proposition \ref{ineq-car-1} gives the inclusion $W^{1,2}( X_c)\subset \BL(  X_c)$. Now pick $f\in \BL( X_c)$ and find a sequence $(f_n)\subset \BL(   X_c)\cap \tilde{\mathcal A}$ as in Proposition \ref{prop-density-2}. By Proposition \ref{eq-car-1} we know that
\[
|\D f_n|_{ X_c}=|\D f_n|_c\qquad \mm_c-a.e.,\ \forall n\in\N.
\]
By construction, the right hand side converges to $|\D f|_c$ in $L^2( X_c,\mm_c)$ as $n\to\infty$, and since $f_n\to f$ in $L^2(  X_c,\mm_c)$, by the lower semicontinuity of weak upper gradients we deduce that $f\in W^{1,2}( X_c)$ and
\[
|\D f|_{ X_c}\leq |\D f|_c,\qquad \mm_c-a.e.,
\]
which together with inequality \eqref{eq:cleqw} gives the thesis.
\end{proof}

\subsection{Warped product}
Throughout this section $w_\d,w_\mm:I\to[0,+\infty)$ are given continuous functions and $X$ is assumed to have finite measure. We are interested in studying Sobolev functions on the warped product space $(X_w,\d_w,\mm_w)$, where $X_w:=I\times_w X$.

Like in the Cartesian case, given $f:X_w\to\R$ and $t\in I$ we shall denote by $f^{(t)}:X\to\R$ the function given by $f^{(t)}(x):=f(t,x)$. Similarly $f^{(x)}(t):=f(t,x)$ for $x\in X$.

We then consider the   Beppo-Levi space  $\BL( X_w)$ defined as follows:

\begin{definition} [The space $\BL(X_w)$]\label{def-warped-BL} As a set,  $\BL( X_w)$ is the subset of $L^2(X_w,\mm_w)$ made of those  functions $f$ such that:
\begin{itemize}
\item[i)] for $\mm$-a.e.\ $x\in X$ we have $f^{(x)} \in W^{1,2}(\R, w_\mm\mathcal L^1)$,
\item[ii)] for ${w_\mm}\mathcal L^1$-a.e. $t\in \R$ we have $f^{(t)}\in W^{1,2}(X)$,
\item[iii)] the function
\begin{equation}
\label{eq:warpedgrad}
|\D f|_w(t,x):=\sqrt{w_d^{-2}(t)|\D f^{(t)}|_X^2(x)+|\D f^{(x)}|^2_\R(t)}
\end{equation}
belongs to $ L^2( X_w,\mm_w)$.
\end{itemize}
On $\BL(X_w)$ we put the norm
\[
\|f\|_{\BL(X_w)}:=\sqrt{\|f\|_{L^2(X_w)}^2+\||\D f|_w\|_{L^2(X_w)}^2}.
\]
\end{definition}
It will be useful to introduce the following auxiliary space:
\begin{definition}[The space $\BL_0(X_w)$]
Let $V\subset \BL(X_w)$ be the space of functions $f$ which are identically 0 on $\Omega\times X\subset X_w$ for some open set $\Omega\subset \R$ containing $\{w_\mm=0\}$.

$\BL_0(X_w)\subset \BL(X_w)$ is defined as the closure of $V$ in $\BL(X_w)$.
\end{definition}
The goal of this section is to compare the spaces $\BL(X_w)$ and $W^{1,2}(X_w)$ and their respective notions of minimal weak upper gradients, namely $|\D f|_w$ and $|\D f|_{X_w}$. Under the sole continuity assumption of $w_d,w_\mm$ and the compatibility condition $\{w_d=0\}\subset \{w_\mm=0\}$ we can prove that
\[
\BL_0(X_w)\subset W^{1,2}(X_w)\subset \BL(X_w)
\]
and that for any $f\in W^{1,2}(X_w)\subset \BL(X_w)$ the identity
\[
|\D f|_{X_w}=|\D f|_w
\]
holds $\mm_w$-a.e., so that in particular the above inclusions are continuous. Without additional hypotheses it is unclear to us whether $W^{1,2}(X_w)=\BL(X_w)$ (on the other hand, it is easy to construct examples where $\BL_0(X_w)$ is strictly smaller than $\BL(X_w)$). Still, if we assume that 
\begin{equation}
\label{eq:disczero}
\text{the set $\{w_\mm=0\}\subset I$ is discrete}
\end{equation}
and that $w_\mm$ decays at least linearly near its zeros, i.e.
\begin{equation}
\label{eq:lindecay}
w_\mm(t)\leq C\inf_{s:w_\mm(s)=0}|t-s|,\qquad\forall t\in\R,
\end{equation}
for some constant $C\in\R$, then we can prove - using basically arguments about capacities - that
\[
\BL_0(X_w)=\BL(X_w),
\]
so that the three spaces considered are all equal. We remark that these two additional assumptions on $w_\mm$ are satisfied in all the geometric applications we have in mind, because typically one considers cone/spherical suspensions and in these cases $w_\mm$ has at most two zeros and decays polynomially near them.

\bigskip

We turn to the details. The following result is easily established:
\begin{proposition}\label{wp} We have  $ W^{1,2}( X _w)\subset\BL( X _w)$.
\end{proposition}
\begin{proof}
Pick $f \in W^{1,2}(X_w)$ and use Proposition \ref{prop-density} to find a sequence $(f_n)$ of Lipschitz functions  on $X_w$ such that $f_n \rightarrow f$ and $\lip{f_n} \rightarrow |\D f|_{X_w}$ in $L^2(X_w)$.  Up to passing to  a subsequence, not relabelled, we can further assume that $\sum_n\|f_{n+1}-f\|_{L^2(X_w)}<\infty$, so that the inequality
\[
\begin{split}
\big\|\sum_n\|f^{(t)}_{n}-f^{(t)}\|_{L^2(X,\mm)}\big\|_{L^2(w_\mm\mathcal L^1)}&=\sqrt{\int_I\int_X\sum_n|f^{(t)}_{n}(x)-f^{(t)}(x)|^2\,\d\mm(x)w_\mm(t)\,\d t}\\
&=\big\|\sum_n|f_n-f|\big\|_{L^2(X_w,\mm_w)}\leq\sum_n \big\||f_n-f|\big\|_{L^2(X_w,\mm_w)}<\infty
\end{split}
\]
shows that for $w_\mm\mathcal L^1$-a.e.\ $t$ we have  $\sum_n\|f^{(t)}_{n}-f^{(t)}\|_{L^2(X,\mm)}<\infty$ and thus in particular  $f^{(t)}_n\to f^{(t)}$ in $L^2(X,\mm)$. Similarly,  
 for $\mm$-a.e. $x \in X$, we have $f^{(x)}_n \rightarrow f^{(x)}$ in $L^2(I,w_\mm\mathcal L^1)$.
 
Observe that for every $(t,x)\in X_w$ we have
\begin{eqnarray*}
\lip {f_n}(t,x)&= & \lmts{(s,y)}{(t,x)} \frac{|f_n(s,y)-f_n(t,x)|}{\d_w({(s,y)},{(t,x)})} \\&\geq&  \lmts{s}{t} \frac{|f_n(s,x)-f_n(t,x)|}{\d_w({(s,x)},{(t,x)})} \\&=&  \lmts{s}{t} \frac{|f^{(x)}_n(s)-f^{(x)}_n(t)|}{|s-t|} = \loclip {f^{(x)}_n}{I}(t)
\end{eqnarray*}
and therefore by Fatou's lemma we deduce
\[
\begin{split}
\int_X\limi_{n\to\infty}\int_I\loclip {f^{(x)}_n}{I}^2(t)\,\d(w_\mm\mathcal L^1)(t)\,\d\mm(x)&\leq \limi_{n\to\infty}\int_{X_w}\lip {f_n}^2(t,x)\,\d\mm_w(t,x)\\
&=\int_{X_w}|\D f|_{X_w}^2\,\d\mm_w<\infty.
\end{split}
\]
Since $f^{(x)}_n \rightarrow f^{(x)}$ in $L^2(I,w_\mm\mathcal L^1)$ for $\mm$-a.e.\ $x\in X$, this last inequality together with the lower semicontinuity of minimal weak upper gradients ensures that $f^{(x)}\in W^{1,2}(I,w_\mm\mathcal L^1)$ for $\mm$-a.e.\ $x\in X$ and
\begin{equation}
\label{eq:tpart}
\int_{X_w}|\D f^{(x)}|_I^2(t)\,\d\mm_w(t,x)\leq \int_{X_w}|\D f|_{X_w}^2\,\d\mm_c.
\end{equation}
An analogous argument starting from the bound
\begin{eqnarray*}
\lip {f_n}(t,x)&=&\lmts{(s,y)}{(t,x)} \frac{|f_n(s,y)-f_n(t,x)|}{\d_w({(s,y)},{(t,x)})}  \\&\geq & \lmts{y}{x} \frac{|f_n(t,y)-f_n(t,x)|}{\d_w({(t,y)},{(t,x)})} \\& = &  \lmts{y}{x} \frac{|f^{(t)}_n(y)-f^{(t)}_n(x)|}{w_\d(t)\d(x,y)} 
=\frac 1{w_\d(t)}\loclip {f^{(t)}_n}{X}(x)
\end{eqnarray*}
valid for every $t\in I$ such that $w_\d(t)>0$, grants that $f^{(t)}\in W^{1,2}(X)$ for $w_\mm\mathcal L^1$-a.e.\ $t\in I$ (recall that $\{w_\d=0\}\subset\{w_\mm=0\}$) and
\begin{equation}
\label{eq:xpart}
\int_{X_w}\frac{|\D f^{(t)}|_X^2(x)}{w_\d^2(t)}\,\d\mm_w(t,x)\leq \int_{X_w}|\D f|_{X_w}^2\,\d\mm_w.
\end{equation}
The bounds \eqref{eq:tpart} and \eqref{eq:xpart} ensure that  $f\in \BL(X_w)$, so that the inclusion $W^{1,2}(X_w)\subset \BL(X_w)$ is proved.
\end{proof}

In order to prove that for  $f\in W^{1,2}(X_w)\subset\BL(X_w)$ the minimal weak upper gradient $|\D f|_{X_w}$ coincides with the `warped' gradient $|\D f|_w$ defined in \eqref{eq:warpedgrad}, we shall make use of the following simple comparison argument, which will then allow us to reduce the proof to the already known cartesian case.
\begin{lemma}\label{gradient-compare}
Let $X$ be a set, $\d_1,\d_2$ two distances on it and $\mm_1,\mm_2$ two measures. Assume that $(X,\d_1,\mm_1)$ and $(X,\d_2,\mm_2)$ are both metric measure spaces satisfying the Assumptions \ref{assumption}, that for some $C>0$ we have  $ \mm_2 \leq C\mm_1$ and that for some $L>0$ we have  $\d_1\leq L \d_2$.

Then denoting by $S(X_1),S(X_2)$ the Sobolev classes relative to  $(X,\d_1,\mm_1)$ and $(X,\d_2,\mm_2)$ respectively and by $|\D f|_1,|\D f|_2$ the associated minimal weak upper gradients, we have  
\[
S(X_1)\subset S(X_2)
\] 
and for every $f\in S(X_1)$ the inequality
\[
|\D f|_{2}\leq L |\D f|_1,
\]
holds $\mm_2$-a.e..
\end{lemma}
\begin{proof}
The assumptions ensure that the topology induced by $\d_2$ is finer than the one induced by $\d_1$, hence every $\d_1$-Borel function is also $\d_2$-Borel.  Then observe that the assumption $\d_1\leq L \d_2$ ensures that $\d_2$-absolutely continuous curves are also $\d_1$-absolutely continuous, the $\d_1$-metric speed being bounded by $L$-times the $\d_2$-metric speed. Then considering also the assumption $ \mm_2 \leq C\mm_1$ we see that $(X,\d_2,\mm_2)$ -test plans are also $(X,\d_1,\mm_1)$-test plans, which, by definition, gives the inclusion $S(X_1)\subset S(X_2)$. The inequality $|\D f|_{2}\leq L |\D f|_1$ $\mm_2$-a.e. is then obtained by the $\mm_2$-a.e.\ minimality of $|\D f|_2$ and the  opposite inequality valid for the metric speeds.
\end{proof}

We can then prove the following result:
\begin{proposition}\label{prop:samegrad} Let $f\in W^{1,2}(X_w)\subset \BL(X_w)$. Then
\[
|\D f|_{X_w}=|\D f|_w,\qquad\mm_w-a.e..
\]
\end{proposition}
\begin{proof}
Fix $\epsilon>0$ and $t_0\in\R$ such that $w_\mm(t_0)>0$ so that also $w_\d(t_0)>0$. Use the continuity of $w_d$ to find $\delta>0$ so that
\begin{equation}
\label{eq:fordistance}
\Big|\frac{w_\d(t)}{w_\d(s)}\Big|\leq 1+\epsilon\qquad \forall t,s\in [t_0-2\delta,t_0+2\delta]
\end{equation}
and let $\nchi:\R\to[0,1]$ be a Lipschitz function identically 1 on $[t_0-\delta,t_0+\delta]$ with support contained in $[t_0-2\delta,t_0+2\delta]$. 

We introduce the continuous functions $\bar w_\d,\bar w_\mm:\R\to \R$ as
\[
\bar w_\d(t):=\left\{\begin{array}{ll}
w_\d(t_0-2\delta),&\quad\text{if }t<t_0-2\delta,\\
w_\d(t),&\quad\text{if }t\in[t_0-2\delta,t_0+2\delta],\\
w_\d(t_0+2\delta),&\quad\text{if }t>t_0+2\delta,\\
\end{array}\right.
\]
\[\bar w_\mm(t):=\left\{\begin{array}{ll}
w_\mm(t_0-2\delta),&\quad\text{if }t<t_0-2\delta,\\
w_\mm(t),&\quad\text{if }t\in[t_0-2\delta,t_0+2\delta],\\
w_\mm(t_0+2\delta),&\quad\text{if }t>t_0+2\delta,\\
\end{array}\right.
\]
the corresponding product space $(X_{\bar w},\d_{\bar w},\mm_{\bar w})$ and consider the function $\bar f:X_w\to\R$ given by $\bar f(t,x):=\nchi(t) f(t,x)$ which belongs to $W^{1,2}(X_w)$ and therefore, by what we just proved, to $\BL(X_w)$. The locality property of minimal weak upper gradients ensures that
\[
|\D f|_{X_w}=|\D \bar f|_{X_w}\quad\text{and}\quad|\D f|_{w}=|\D \bar f|_{w}\qquad\mm_w-a.e.\ \text{on }[t_0-\delta,t_0+\delta]\times X.
\]
Since $\bar f$ has support concentrated in the set of $(t,x)$'s with $t\in[t_0-2\delta,t_0+2\delta]$ and $w_\d$ is positive in such interval, we can think of $\bar f$ also as a real valued function $X_{\bar w}$. With this identification in mind it is clear that
\[
|\D \bar f|_{X_w}=|\D \bar f|_{X_{\bar w}}\quad\text{and}\quad|\D \bar f|_{w}=|\D \bar f|_{\bar w}\qquad\mm_w-a.e.\ \text{on }[t_0-2\delta,t_0+2\delta]\times X.
\]
We now consider the cartesian product $(X_c,\d_c,\mm_c)$ of $(X,\d,\mm)$ and $\R$. Notice that the sets $X_{\bar w}$ and $X_c$ both coincide with $\R\times X$ and that by construction (recall also \eqref{eq:fordistance}) we have
\[
c\mm_{c}\leq \mm_{\bar w}\leq C\mm_{c} \qquad\text{and}\qquad \frac{w_\d(t_0)}{1+\eps}\d_c\leq \d_{\bar w}\leq w_\d(t_0)(1+\eps)\d_c
\]
where  $c:=\min_{[t_0-2\delta,t_0+2\delta]}w_\mm$ and $C=\max_{[t_0-2\delta,t_0+2\delta]}w_\mm$. Hence by Lemma \ref{gradient-compare} we deduce that $\mm_{\bar w}$-a.e.\ it holds
\[
\frac{|\D \bar f|_{X_c}}{w_\d(t_0)(1+\eps)}\leq |\D \bar f|_{X_{\bar w}}\leq \frac{1+\eps}{w_\d(t_0)}|\D \bar f|_{X_c}\quad\text{and}\quad
\frac{|\D \bar f|_{c}}{w_\d(t_0)(1+\eps)}\leq |\D \bar f|_{{\bar w}}\leq \frac{1+\eps}{w_\d(t_0)}|\D \bar f|_{c}.
\]

Since by Theorem \ref{theorem-car} we know that $|\D\bar f|_{X_c}=|\D \bar f|_c$ $\mm_c$-a.e.,  collecting what we proved we deduce that
\[
\frac{|\D f|_{X_w}}{(1+\eps)^2}\leq |\D f|_w\leq (1+\eps)^2|\D f|_{X_w}
\]
$\mm_w$-a.e.\ on $[t_0-\delta,t_0+\delta]\times X$. By the arbitrariness of $t_0$ such that $w_\mm(t_0)>0$  and the Lindelof property of $\{w_\mm>0\}\subset\R$ we deduce that the above inequality holds $\mm_w$-a.e.. The conclusion then follows letting $\eps\downarrow0$.
\end{proof}
We now turn to the general relation between $\BL_0(X_w)$ and $W^{1,2}(X_w)$:
\begin{proposition} We have  $\BL_0(X_w)\subset W^{1,2}(X_w)$.
\end{proposition}
\begin{proof}
Taking into account Proposition \ref{prop:samegrad} it is sufficient to prove that $V\subset W^{1,2}(X_w)$. Notice that for arbitrary $f\in\BL(X_w)$, considering the functions  $\nchi_n(t):=0\vee (n-|t|)\wedge 1$ and defining $f_n(t,x):=\nchi_n(t)f(t,x)$, via a direct verification of the definitions we have $f_n\in\BL(X_w)$, while inequality \eqref{eq:leibn} and the dominated convergence theorem grant that $f_n\to f$ in $\BL(X_w)$. Therefore, using again Proposition \ref{prop:samegrad} which ensures that $\BL$-convergence implies $W^{1,2}$-convergence, to conclude it is sufficient to show that any $f\in V$ with support contained in $(I\cap [-T,T])\times X\subset X_w$ for some $T>0$ belongs to $W^{1,2}(X_w)$.

Thus fix such  $f\in V$, for $r>0$ denote by $\Omega_r\subset \R$ the $r$-neighborhood of  $\{w_\mm=0\}$ and find $r\in(0,1)$ such that $f$ is $\mm_w$-a.e.\ zero on $\Omega_{2r}\times X$.  Then by continuity and compactness and recalling that  $\{w_\d=0\}\subset \{w_\mm=0\}$ we deduce that there are constants $0<c\leq C<\infty$ such that
\[
c\leq w_\d(t),w_\mm(t)\leq C,\qquad\forall t\in I\cap [-T,T]\setminus\Omega_{r/2}.
\]
We are now going to use a comparison argument similar to that used in the proof of Proposition \ref{prop:samegrad}. Find two continuous functions $w'_\d,w'_\mm$ agreeing with $w_\d,w_\mm$ on $[-T,T]\setminus\Omega_{r/2}$ and such that $c\leq w'_\d,w'_\mm\leq C$ on the whole $\R$ and consider the warped product $(X_{w'},\d_{w'},\mm_{w'})$ and the cartesian product $(X_c,\d_c,\mm_c)$ of $I$ and $X$. We then have the equalities of sets:
\[
\BL(X_{w'})=\BL(X_c)=W^{1,2}(X_c)=W^{1,2}(X_{w'}),
\]
the first and last  coming from Lemma \ref{gradient-compare} and the properties of $w'_\d,w'_\mm$ and the middle one being given by Theorem \ref{theorem-car}.

By the construction of $w'_\d,w'_\mm$ we see that $f\in \BL(X_{w'})$ and thus, by what we just proved, that $f\in W^{1,2}(X_{w'})$. Then Proposition \ref{prop-density} grants that there exists a sequence $(f_n)$ of $\d_{w'}$-Lipschitz functions converging to $f$ in $L^2(X_{w'})$ with  
\[
\sup_{n\in\N}\int {\rm lip}'(f_n)^2\,\d\mm_{w'}<\infty
\] 
uniformly bounded in $n$, where by ${\rm lip}'$ we denote the local Lipschitz constant computed w.r.t.\ the distance $\d_{w'}$. Notice that up to replacing $f_n$ with $(-C_n)\vee f_n\wedge C_n$ for a sufficiently large $C_n$, we can, and will, assume that $f_n$ is bounded for every $n\in\N$.

Now find a Lipschitz  function $\nchi:I\to [0,1]$ identically 0 on $\Omega_r\cup(I\setminus[-T-1,T+1])$, identically 1 on $I\cap [-T,T]\setminus\Omega_{2r}$ and put $\tilde f_n(t,x):=\nchi(t)f_n(t,x)$. By construction it is immediate to check that the $\tilde f_n$'s are still $\d_{w'}$-Lipschitz, converging to $f$ in $L^2(\mm_{w'})$ and satisfying
\begin{equation}
\label{eq:bound1}
\sup_{n\in\N}\int {\rm lip}'(\tilde f_n)^2\,\d\mm_{w'}<\infty.
\end{equation}
We now claim that the $\tilde f_n$'s are $\d_w$-Lipschitz, converging to $f$ in $L^2(X_w)$ and such that
\begin{equation}
\label{eq:bound2}
\sup_{n\in\N}\int {\rm lip}(\tilde f_n)^2\,\d\mm_{w}<\infty,
\end{equation}
from which the conclusion follows by the lower semicontinuity of weak upper gradients and the bound $|\D \tilde f_n|_{X_w}\leq \lip{\tilde f_n} $ valid $\mm_w$-a.e.. Since all the functions $\tilde f_n$ and $f$ are concentrated on $([-T,T]\setminus\Omega_r)\times X$ and on this set the measures $\mm_w$ and $\mm_{w'}$ agree, we clearly have $L^2(X_w)$-convergence. Moreover, since $w_\d$ and $w_{\d'}$ agree on $([-T,T]\setminus\Omega_r)\times X$, the topologies on $([-T,T]\setminus\Omega_r)\times X$ induced by $\d_w$ and $\d_{w'}$ agree (with the product topology, given that these functions are positive) and a direct use of the definition yields
\[
\lim_{(s,y)\to (t,x)}\frac{\d_w\big((s,y),(t,x)\big)}{\d_{w'}\big((s,y),(t,x)\big)}=1,\qquad\forall (t,x)\in([-T,T]\setminus\Omega_r)\times X.
\] 
In particular, we have ${\rm lip}(\tilde f_n)={\rm lip}'(\tilde f_n)$ in $([-T,T]\setminus\Omega_r)\times X$, so that \eqref{eq:bound2} follows from \eqref{eq:bound1}. It remains to prove that $\tilde f_n$ is $\d_w$-Lipschitz; to this aim recall that on a length space the Lipschitz constant of a function is equal to the supremum of the local Lipschitz constant and conclude by
\[
\Lip(f_n)= \sup_{X_w}\lip{f_n}=\sup_{X_{w'}}{\rm lip}'(f_n)=\Lip'(f_n)<\infty,
\]
where $\Lip'(f_n)$ denotes the $\d_{w'}$-Lipschitz constant.
\end{proof}
Finally, we prove that if the set of zeros of $w_\mm$ is discrete and $w_\mm$ decays at least linearly close to its zeros, then $\BL_0(X_w)=\BL(X_w)$:
\begin{proposition}  Assume that $w_\mm$ has the properties \eqref{eq:disczero}
 and \eqref{eq:lindecay}.
 
Then $\BL_0(X_w)=\BL(X_w)$.
\end{proposition}
\begin{proof}
A standard truncation argument shows that $\BL\cap L^\infty(X_w)$  is dense in $\BL(X_w)$, so to conclude it is sufficient to show that for any $f\in\BL\cap L^\infty(X_w)$ we can find a sequence $(f_n)\subset V$ converging to it in $\BL(X_w)$.

Thus pick $f\in \BL\cap L^\infty(X_w)$, put $D(t):=\min_{s:w_\mm(s)=0}|t-s|$ and for $n,m\in\N$, $n>1$ consider the  cut-off functions 
\[
\begin{split}
\sigma_m(x)&:=0\vee(m-\d(x,\bar x))\wedge 1,\\
\eta_n(t)&:=0\vee\Big( 1-\frac{|\log(D(t))|}{\log (n)}\Big)\wedge 1,\\
\tilde\eta_n(t)&:=0\vee\Big( n-|t|\Big)\wedge 1,
\end{split}
\]
where $\bar x\in X$ is a chosen, fixed point, and  define $f_{n,m}(t,x):=\eta_n(t)\tilde\eta_n(t)\sigma_m(x)f(t,x)$. Since  $(t,x)\mapsto \eta_n(t)\tilde\eta_n(t)\sigma_m(x)$ is Lipschitz and bounded for every $n,m$, a direct check of the definition of $\BL(X_w)$ shows that $f_{n,m}\in \BL(X_w)$ for every $n,m$ and, since $\eta_n$ is 0 on a neighborhood of $\{w_\mm=0\}$, we also have $f_{n,m}\in V$ for every $n,m$.

Using the fact that the functions   $(t,x)\mapsto \eta_n(t)\tilde\eta_n(t)\sigma_m(x)$ are uniformly bounded by 1 and pointwise converge to 1 as $n,m\to\infty$ and the dominated convergence theorem we see that  $f_{n,m}\to f$ in $L^2(X_w)$ as $n,m\to\infty$.

Next, recalling \eqref{eq:leibn} and using that $\sigma_m$ is 1-Lipschitz we see that
\[
|\D (f^{(t)}-f^{(t)}_{n,m})|_X(x)\leq \big|\eta_n(t)\tilde\eta_n(t)\sigma_m(x)-1\big|\,|\D f^{(t)}|_X(x)+|f(t,x)|1_{\{\d(\cdot,\bar x)\geq m-1\}}(x)
\]
for $\mm_w$-a.e.\ $(t,x)$, so that the dominated convergence theorem again gives that $\int |\D(f^{(t)}-f^{(t)}_{n,m})|_X^2(x)\,\d\mm_w(t,x)\to  0$ as $n,m\to\infty$.

Similarly, we have
\[
\begin{split}
|\D(f^{(x)}-f^{(x)}_{n,m})|_I(t)\leq&  \big|\eta_n(t)\tilde\eta_n(t)\sigma_m(x)-1\big|\,|\D f^{(x)}|_I(t)+|f(t,x)| 1_{\{|\cdot|\geq n-1\}}(t)\\
&+|f(t,x)| 1_{\{\d(\cdot,\bar x)\leq m\}}(x) 1_{\{|\cdot|\leq n\}}(t)|\partial_t\eta_n|(t)
\end{split}
\]
for $\mm_w$-a.e.\ $(t,x)$ and again by dominated convergence we see that the first two terms in the right hand side go to 0 in $L^2(X_w)$ as $n,m\to \infty$. For the last term, we use the fact that $f$ is bounded and our  assumptions on $w_\mm$. Observe indeed that $|\partial_t\eta_n|(t)\leq \frac{ 1_{D^{-1}([n^{-1},1])}(t)}{D(t)\log n}$ so that letting $x_1,\ldots,x_N$ be the finite number of zeros of $w_\mm$ in $[-n-1,n+1]$ we have
\[
\begin{split}
\int |f(t,x)|^2 1_{\{\d(\cdot,\bar x)\leq m\}}(x)& 1_{\{|\cdot|\leq n\}}(t)|\partial_t\eta_n|^2(t) \,\d\mm_w\\
&\leq \frac{\|f\|^2_{L^\infty}\mm(B_m(\bar x))}{\log(n)^2}\int_{[-n,n]\cap D^{-1}([n^{-1},1]) }\frac{1}{D^2(t)}w_\mm(t)\,\d t\\
&\leq C\frac{\|f\|^2_{L^\infty}\mm(B_m(\bar x))}{\log(n)^2}\int_{[-n,n]\cap D^{-1}([n^{-1},1])}\frac{1}{D(t)}\,\d t\\
&\leq C\frac{\|f\|^2_{L^\infty}\mm(B_m(\bar x))}{\log(n)^2}\sum_{i=1}^N\int_{\{t:|t-x_i|\in[n^{-1},1]\}}\frac{1}{|t-x_i|}\,\d t\\
&= 2NC\frac{\|f\|^2_{L^\infty}\mm(B_m(\bar x))}{\log(n)}.
\end{split}
\]
Since the last term goes to 0 as $n\to\infty$ for every $m\in\N$, we just proved that
\[
\lim_{m\to\infty}\lim_{n\to\infty}\int |\D(f^{(x)}-f^{(x)}_{n,m})|_I^2(t)\,\d\mm_w(t,x)=0,
\]
which is sufficient to conclude.
\end{proof}

\subsection{Sobolev-to-Lipschitz property}\label{se:stl}
The aim of this section is to study the Sobolev-to-Lipschitz property on warped products; let us recall the definition:
\begin{definition}[Sobolev-to-Lipschitz property]\label{def:stl}
We say that   a metric measure space $(X ,\d, \mm)$ has  Sobolev to Lipschitz property if for any function $f\in W^{1,2}(X)$  with $|\D f|_X \in L^\infty(X)$, we can find a function  $\tilde{f}$ such that $f=\tilde{f}$ $\mm$-a.e. and  $\Lip (\tilde f)=\esup ~{|\D f|_X}$.
\end{definition}

\bigskip

Metric measure spaces with the Sobolev-to-Lipschitz property are, in some sense, those whose metric properties can be studied via Sobolev calculus. Such property firstly appeared, in its `dual' formulation and in the formalism of Dirichlet forms, in \cite{AmbrosioGigliSavare11-2}, where it was written as:
\begin{equation}
\label{eq:stlrcd}
\begin{split}
&\text{The intrinsic distance induced by the quadratic Cheeger energy}\\
&\text{is equal to the distance on the space}
\end{split}
\end{equation}
and it was proved that $\rcd$ spaces have such property. \eqref{eq:stlrcd} has also been used in \cite{AmbrosioGigliSavare12} as one of the ingredients for the axiomatization of the $\rcd$ condition purely in terms of Dirichlet forms. In \cite{Gigli13} the definition as above has been introduced and it has been observed that, regardless of lower Ricci bounds or infinitesimal Hilbertianity, this notion is sufficient to derive metric properties of the space out of properties of Sobolev functions defined on it. As discussed in the introductionthis is relevant in applications when little to nothing is known about the base space.

We observe that under the only assumption that $w_\d,w_\mm$ are continuous we cannot hope to prove that $X_w$ has the Sobolev-to-Lipschitz property. Indeed, if $w_\mm$ is 0 on some subinterval of $I$ which disconnects $I$, then the measure $\mm_w$ has disconnected support and therefore we can find  non-constant functions on $X_w$ which are locally constant on the support of $\mm_w$, which is easily seen to violate the Sobolev-to-Lipschitz condition.

We shall therefore only consider the case where $w_\mm$ is strictly positive in the interior of $I$, a condition which is satisfied in the standard geometric constructions like that of cone/spherical suspension.

It is unclear to us whether, even with this condition on $w_\mm$, the Sobolev-to-Lipschitz property passes to warped products or not. What we are able to do, instead, is to identify two quite general properties which imply the Sobolev-to-Lipschitz, each of which passes to warped products. This will also imply that whenever these two properties hold for a certain base space, one can consider multiple warped products and still get that the final space has the Sobolev-to-Lipschitz property.

\bigskip

We begin introducing the two auxiliary concepts we just alluded to. The first is a variant  of the length property which takes into account the reference measure:
\begin{definition}[Measured-length space]\label{def-good}
We say that a metric measure space $(X,\d,\mm)$ is measured-length if there exists a Borel set $A\subset X$ whose complement is $\mm$-negligible with the following property. For every $x_0,x_1\in A$  there exists $ \eps>0$ such that for every  $\eps_0,\eps_1\in(0,\eps]$ there is a test plan $\pi^{\eps_0,\eps_1}\in\mathcal P(C([0,1],X))$ with:
\begin{itemize}
\item[a)] the map $(0,\eps]^2\ni(\eps_0,\eps_1)\mapsto \pi^{\eps_0,\eps_1}$ is weakly Borel in the sense that for any $\varphi\in C_b(C([0,1],X))$ the map
\[
(0,\eps]^2\ni(\eps_0,\eps_1)\qquad\mapsto \qquad \int \varphi\,\d\pi^{\eps_0,\eps_1},
\]
is Borel.
\item[b)] We have 
\[
(e_0)_\sharp\pi^{\eps_0,\eps_1}= \frac{1_{B_{\eps_0}(x_0)}}{\mm(B_{\eps_0}(x_0))}\,\mm,\qquad\text{ and }\qquad (e_1)_\sharp\pi^{\eps_0,\eps_1}= \frac{1_{B_{\eps_1}(x_1)}}{\mm(B_{\eps_1}(x_1))}\,\mm
\]
 for every $\eps_0,\eps_1\in(0,\eps]$,
\item[c)] We have
\begin{equation}
\label{eq:limitenergy}
\lims_{\eps_0,\eps_1\downarrow 0}\iint_0^1|\dot\gamma_t|^2\,\d t\,\d\pi^{\eps_0,\eps_1}(\gamma)\leq \d^2(x_0,x_1).
\end{equation}
\end{itemize}
\end{definition}
\begin{remark}\label{rem:constsp}{\rm
Notice that $(b)$ forces
\[
\limi_{\eps_0,\eps_1\downarrow 0}\iint_0^1|\dot\gamma_t|^2\,\d t\,\d\pi(\gamma)\geq \d^2(x_0,x_1),
\]
so that $(c)$ can be read as a sort of `optimality in the limit'. In this direction, notice that with a reparametrization by constant-speed argument one can always assume that the plans $\pi$ above also fulfil:
\begin{equation}
\label{eq:constdef}
t\qquad\mapsto\qquad\int |\dot\gamma_t|^2\,\d\pi(\gamma)\qquad\text{ is constant,}
\end{equation}
because such reparametrization decreases the kinetic energy.
}\fr
\end{remark}
The second definition is a simple modification of the usual doubling notion:
\begin{definition}[a.e.\ locally doubling spaces]\label{def:aedoub}
We say that a metric measure space $(X,\d,\mm)$ is a.e.\ locally doubling provided there exists a Borel set $B$ whose complement is $\mm$-negligible such that for every $x\in B$ there are an open set $\Omega$ containing $x$  and constants $C,R>0$ such that
\[
\mm(B_{2r}(y))\leq C\mm(B_r(y)),\qquad\forall r\in(0,R),\ y\in \Omega.
\]
\end{definition}
It is easy to check that a a.e.\ locally doubling and measured-length space has the Sobolev-to-Lipschitz property:
\begin{proposition}\label{prop:stl}
Let $(X,\d,\mm)$ be an a.e.\  locally doubling and measured-length space. Then it has the Sobolev-to-Lipschitz property.
\end{proposition}
\begin{proof} It is well known that on doubling spaces, for any given function in $L^1_{loc}$, a.e.\ point is a Lebesgue point. Since the property of being a Lebesgue point is local in nature,  we immediately have that even on a a.e.\ locally doubling space a.e.\ point is a Lebesgue point of a given $L^1_{loc}$ function.

With that said, let $A\subset X$ be the set given in Definition \ref{def-good},  pick $f\in W^{1,2}(X)$  with $L:=\esup~{|\D f|} < \infty$ and let $B\subset X$ the set of its Lebesgue points of $f$.

Pick $x,y\in A\cap B$, let $\eps$ be given by Definition \ref{def-good}, consider $\eps'\in(0,\eps]$ and the test plan $\pi^{\eps',\eps'}$ given by Definition \ref{def-good}. Then we have
\[
\begin{split}
\int|f(\gamma_1)-f(\gamma_0)|\,\d\pi^{\eps',\eps'}(\gamma)&\leq \iint_0^1|\D f|(\gamma_t)|\dot\gamma_t|\,\d t\,\d\pi^{\eps',\eps'}(\gamma)\\
&\leq L \iint_0^1|\dot\gamma_t|\,\d t\,\d\pi^{\eps',\eps'}(\gamma)\leq L\sqrt{ \iint_0^1|\dot\gamma_t|^2\,\d t\,\d\pi^{\eps',\eps'}(\gamma)}
\end{split}
\]
Letting $\eps'\downarrow0$, using the fact that $x,y$ are Lebesgue points and the bound \eqref{eq:limitenergy} we obtain
\[
\begin{split}
|f(y)-f(x)|&=\lim_{\eps'\downarrow0}\left|\int f\,\d(e_0)_\sharp\pi^{\eps',\eps'}-\int f\,\d(e_1)_\sharp\pi^{\eps',\eps'}\right|\\
&=\lim_{\eps'\downarrow0}\left|\int f(\gamma_1)-f(\gamma_0)\,\d\pi^{\eps',\eps'}(\gamma)\right|\\
&\leq \limi_{\eps'\downarrow0}\int |f(\gamma_1)-f(\gamma_0)|\,\d\pi^{\eps',\eps'}(\gamma)\\
&\leq L\limi_{\eps'\downarrow0}\sqrt{ \iint_0^1|\dot\gamma_t|^2\,\d t\,\d\pi^{\eps',\eps'}(\gamma)}\leq L\d(x,y).
\end{split}
\]
This proves that the restriction of $f$ to $A\cap B$ is Lipschitz with Lipschitz constant bounded by $L$. Since $A\cap B$ has full measure the proof is achieved.
\end{proof}
We shall now verify that both the properties of being a.e.\ locally doubling and measured-length pass to warped products. We start with the a.e.\ locally doubling, which is easier.
\begin{proposition}\label{prop:dw}
Let $(X,\d,\mm)$ be an a.e.\ locally doubling space  with finite measure and $w_\d,w_\mm:\to[0,+\infty)$ continuous functions. Then the warped product $(X_w,\d_w,\mm_w)$ is a.e.\ locally doubling as well.
\end{proposition}
\begin{proof}
Let $B\subset X$ be the set given in Definition \ref{def:aedoub}, put $\hat B:=\{w_\mm>0\}\times B\subset X_w$ and notice that $\hat B$ has negligible complement. Then recall that $\{w_\mm>0\}\subset \{w_\d>0\}$, notice that it is trivial that the cartesian product of a doubling space and an interval is doubling and conclude using the continuity of $w_\d,w_\mm$. 
\end{proof}
To study the behavior of the measured-length property on warped products, we need to recall some facts about warped product distances. The content of the following lemma is well known, but we provide the simple proof for completeness.
\begin{lemma}\label{le:balls}
Let $(X,\d)$ be a complete and separable length space, $I\subset\R$ a closed, possibly unbounded interval, $w_\d:I\to\R^+$ a continuous function and consider the warped product metric space $(X_w,\d_w)$. 

Then there exists a function $D:I^2\times X\to\R^+$ such that
\begin{equation}
\label{eq:distD}
\d_w\big((t_0,x_0),(t_1,x_1)\big)=D(t_0,t_1,\d(x_0,x_1)),\qquad\forall t_0,t_1\in I,\ x_0,x_1\in X,
\end{equation}
and for every sequence of curves $\gamma_n=(\gamma^I_n,\gamma^X_n)$ joining $(t_0,x_0)$ to $(t_1,x_1)$ whose $\d_w$-length converge to $\d_w\big((t_0,x_0),(t_1,x_1)\big)$ we have that the $\d$-length of the curves $\gamma^X_n$ converge to $\d(x_0,x_1)$.

Finally, for $(t_0,x_0)\in X_w$ with $w_\d(t_0)>0$ and given $\eps>0$ and $t'\in I$ we have that there exists $x'\in X$ with $ (t',x') \in B_\eps((t_0,x_0))$ if and only if $|t'-t|<\eps$ and in this case the set of such $x'$'s is a ball centered at $x_0$ whose $\d$-radius $r(t_0,t',\eps)$ satisfies
\begin{equation}
\label{eq:smallradius}
\lims_{\eps\downarrow 0}\sup_{t'\in [t_0-\eps,t_0+\eps]}r(t_0,t',\eps)=0.
\end{equation}
\end{lemma}
\begin{proof}
Fix $t_0,t_1\in I$, $x_0,x_1\in X$ and let $\Gamma_w\subset C([0,1],X_w)$ be the set of absolutely continuous curves joining $(t_0,x_0)$ to $(t_1,x_1)$ and $\Gamma\subset C([0,1],I)$ the set of absolutely continuous curves joining $t_0$ to $t_1$. Also, define $L_w:\Gamma_w\to\R^+$ and $L:\Gamma\to\R^+$ as
\[
\begin{split}
L_w((\gamma^I,\gamma^X))&:=\int_0^1\sqrt{|\dot\gamma^I_s|^2+w_\d^2(\gamma^I_s)|\dot\gamma^X_s|^2}\,\d s,\\
L(\gamma^I)&:=\int_0^1\sqrt{|\dot\gamma^I_s|^2+w_\d^2(\gamma^I_s)\d^2(x_0,x_1)}\,\d s
\end{split}
\]
and notice that $L_w$ is invariant under reparametrization. We claim that 
\begin{equation}
\label{eq:forD}
\inf_{\gamma\in \Gamma_w}L_w(\gamma)=
\inf_{\gamma^I\in\Gamma}L(\gamma^I).
\end{equation}
Indeed, to get inequality $\leq$ find a sequence of curves $\gamma^X_n$ joining $x_0$ to $x_1$ parametrized with constant speed and whose length converges to $\d(x_0,x_1)$. Then for every $\gamma^I\in\Gamma$ consider the curves $\gamma_n:=(\gamma^I,\gamma^X_n)\in\Gamma_w$ and notice that $\lim_nL_w(\gamma_n)=L(\gamma^I)$.

To prove $\geq$, pick $\gamma=(\gamma^I,\gamma^X)\in\Gamma_w$ and up to a small perturbation which does not alter $L_w$ much, assume that the curve $\gamma^X$ has always positive speed. Then let $\tilde\gamma=(\tilde\gamma^I,\tilde\gamma^X)\in\Gamma_w$ be the reparametrization of $\gamma$ chosen so that $\tilde\gamma^X$ has constant speed, call it $\ell$. Then we have $\ell\geq \d(x_0,x_1)$ and thus
\[
L_w(\gamma)=L_w(\tilde\gamma)=\int_0^1\sqrt{|\dot{\tilde\gamma}^I_s|^2+w_\d^2(\gamma^I_s)\ell^2}\,\d s\geq \int_0^1\sqrt{|\dot{\tilde\gamma}^I_s|^2+w_\d^2(\gamma^I_s)\d^2(x_0,x_1)}\,\d s= L(\tilde\gamma^I),
\]
concluding the proof of \eqref{eq:forD}.

The identity \eqref{eq:forD} gives the existence of the function $D$ claimed in the statement, because the left-hand side of \eqref{eq:forD} equals $\d_w\big((t_0,x_0),(t_1,x_1)\big)$ while the right-hand side depends on $t_0,t_1$ and $\d(x_0,x_1)$ only.

The same arguments just used also yield the claim about the convergence of the $\d$-length of the curves $\gamma^X_n$.

Concerning the last statement, notice that $\d_w((t_0,x_0),(t',x'))\geq |t_0-t'|$ for every $t_0,t'\in I$ and $x_0,x'\in X$ and that the equality holds if $x_0=x'$. This addresses the claim about the existence of $x'$. The fact that the set of such $x'$'s is a ball follows directly from \eqref{eq:distD} so that it remains to prove  \eqref{eq:smallradius}. Notice that $r(t_0,t',\eps)$ is characterized by the identity
\[
D\big(t_0,t',r(t_0,t',\eps)\big)=\eps
\]
and that from \eqref{eq:forD} and the definition of $L$ we see, choosing $\gamma_s^I:=(1-s)t_0+st'$, that
\[
D\big(t_0,t',r(t_0,t',\eps)\big)\geq \sqrt{|t_0-t'|^2+|r(t_0,t',\eps)|^2\inf_{[t_0-\eps,t_0+\eps]}w^2_\d}\geq r(t_0,t',\eps)\inf_{[t_0-\eps,t_0+\eps]}w_\d.
\] 
Hence
\[
r(t_0,t',\eps)\leq\frac{ \eps }{\inf_{[t_0-\eps,t_0+\eps]}w_\d}
\]
and the conclusion follows by the continuity of $w_\d$ and the assumption $w_\d(t_0)>0$.
\end{proof}

We turn to the proof that the warped product of an interval and a measured-length space is still measured-length. Unfortunately, the argument is a bit tedious: in the course of the proof, after having introduced some key objects, we shall explain what is the basic idea for the construction.

\begin{proposition}\label{prop:mlw}
Let $(X,\d,\mm)$ be a measured-length space of finite mass, $I\subset \R$ a closed, possibly unbounded, interval and $w_\d,w_\mm:I\to[0,+\infty)$  continuous  functions. Assume that $w_\mm$ is strictly positive in the interior of $I$.

Then the warped product $(X_w,\d_w,\mm_w)$ is a measured-length space as well.
\end{proposition}
\begin{proof} \ \\
\noindent{\bf Step 1: set up of the construction.} Let $A\subset X$ be the set given in the definition of measured-length space, $\mathring I$ the interior of $I$ and put $\hat A:=\mathring I\times A\subset X_w$. Notice that $\hat A$ has full $\mm_w$-measure and fix $(t_0,x_0),(t_1,x_1)\in \hat A$.

We assume for the moment that there is a $\d_w$-geodesic $\gamma=(\gamma^I,\gamma^X)$ connecting $(t_0,x_0)$ to $(t_1,x_1)$ such that the image of $\gamma^I$, which we shall call $K$, is contained in $\mathring I$. Without loss of generality, we shall assume that $\gamma$ has constant speed, so that
\[
\int_0^1|\dot\gamma_s|^2\,\d s=\d_w^2\big((t_0,x_0),(t_1,x_1)\big).
\]
Since $K$ is compact, for some $\delta>0$ its $\delta$-neighborhood $K_\delta$  is still contained in  $\mathring I$ and thus the quantities $\inf_{K_{\delta/2}} w_\d$ and $\inf_{K_{\delta/2}} w_\mm$ are strictly positive. 

As will be clear in a moment, all the objects that we are going to build will live in  $K_{\delta/2}\times X\subset X_w$ and since the role of the reference measure $\mm_w$ is only for the $L^\infty$-bound in (c) of Definition \ref{def-good}, and since we have 
\[
\mm_c\restr{K_{\delta/2}\times X}\leq \frac{\mm_w\restr{K_{\delta/2}\times X}}{\inf_{K_{\delta/2}}w_\mm}
\] without loss of generality, we may replace $\mm_w$ by   $\mm_c$.

Let $ \eps$ be the number given  in Definition \ref{def-good} related to the space $X$ and the points $x_0,x_1\in A$, put $\hat\eps:=\min\{\frac\delta4,\eps\inf_{K_{\delta/2}}w_\d\}$ and fix $\eps_0,\eps_1\in(0,\hat\eps]$. Most of the objects we shall define from now on will depend on $\eps_0,\eps_1$, but to keep the notation simple we shall often avoid explicitly referring to them.

Consider the two measures
\[
\mu_{0} = \frac{1_{B_{\eps_0}(t_0,x_0)}}{\mm_c(B_{\eps_0}(t_0,x_0))}\,\mm_c,\qquad\qquad\qquad\qquad \mu_{1} = \frac{1_{B_{\eps_1}(t_1,x_1)}}{\mm_c(B_{\eps_1}(t_1,x_1))}\,\mm_c,
\]
put $\nu_i:=\pi^I_\sharp\mu_i\in\mathcal P(I)$ and let $\{\mu_{i,t}\}_{t\in I}$ be the disintegration of $\mu_i$ w.r.t.\ $\pi^I$, $i=0,1$. We shall think to the measures $\mu_{i,t}$ as measures on $(X,\d,\mm)$ so that the construction and Lemma \ref{le:balls} ensure that
\[
\mu_{i,t}=\frac{1_{B_{f_i(t)}(x_i)}}{\mm(B_{f_i(t)}(x_i))}\mm,\qquad\qquad \forall t\in{\rm spt}(\nu_i)=[t_i-\eps_i,t_i+\eps_i],\quad i=0,1,
\]
for some functions $f_0,f_1$ which, due to the choice of $\hat \eps,\eps_0,\eps_1$ and the fact that we are considering those balls in the non-rescaled space $(X,\d,\mm)$, satisfy $f_0(t),f_1(t)\leq \eps$ for every $t$ in the respective domain of definition. Notice that inequality \eqref{eq:smallradius} gives
\begin{equation}
\label{eq:fi}
\lims_{\eps_i\downarrow0}\sup_{t\in[t_i-\eps_i,t_i+\eps_i]}f_i(t)=0,\qquad\text{ for }i=0,1.
\end{equation}
Observe that $\nu_0,\nu_1\ll\mathcal L^1$, let   $T:I\to I$ be the optimal transport map from $\nu_0$ to $\nu_1$ and for  $t\in[t_0-\eps_0,t_0+\eps_0]$ consider  the plan $\pi^{f_0(t),f_1(T(t))}\in\mathcal P(C([0,1],X))$ joining $\mu_{0,t}$ to $\mu_{1,T(t)}$ whose existence is ensured by Definition \ref{def-good} and the fact that $f_0,f_1\leq\eps $. According to Remark \ref{rem:constsp} we can also assume that 
\begin{equation}
\label{eq:constspe}
s\quad\mapsto\quad\int |\dot\gamma_s|^2\,\d\pi^{f_0(t),f_1(T(t))}(\gamma)\qquad\text{ is constant, call it  }{\rm Sp}^2(t).
\end{equation}
Then we also know that for some Borel function $t\mapsto C(t)\geq1$ we have
\begin{equation}
\label{eq:bounddenst}
(e_s)_\sharp\pi^{f_0(t),f_1(T(t))}\leq C(t)\mm,\qquad\forall s\in[0,1]
\end{equation}
and point $(c)$ in Definition \ref{def-good}, \eqref{eq:fi} and \eqref{eq:constspe} ensure that
\begin{equation}
\label{eq:sppi}
\lims_{\eps_0,\eps_1\downarrow0}\sup_{t\in[t_0-\eps_0,t_0+\eps_0]}{\rm Sp}^2(t)\leq \d^2(x_0,x_1).
\end{equation}
{\bf Interlude: idea of the construction} We shall build a test plan joining $\mu_0$ to $\mu_1$ in a way that resembles Knothe's rearrangement, where - informally speaking - we rearrange the mass on the $I$-component according to the geodesic from $\nu_0$ to $\nu_1$ and, for fixed $t$, the disintegrated measure $\mu_{0,t}$ is sent to $\mu_{1,T(t)}$ via the plan $\pi^{f_0(t),f_1(T(t))}$. This will almost give the desired answer, meaning that the plan built this way would satisfy $(a),(b),(c)$ in Definition \ref{def-good}. 

However, in general the resulting plan won't be a test plan because in principle it might not be of bounded compression:  the problem is that we don't have any form of control on  the constant $C(t)$ in \eqref{eq:bounddenst} in terms of $t$ so that when the plans $\pi^{f_0(t),f_1(T(t))}$ are `glued together' we can't be sure of having a uniform bound on the densities of the marginals.

To overcome this problem, before interpolating in the $X$-variable we shall first slightly modify $\mu_0,\mu_1$ into measures having $I$-marginals small enough to compensate the `exploding' quantity $C(t)$. We can do this with only a minor loss of kinetic energy, so that the new plan will still satisfy $(c)$ in Definition \ref{def-good}, thus leading to the conclusion.\\
{\bf Back to Step 1}
For $\nu_0$-a.e.\ $t$ let  $\mathfrak s:[0,1]\to [0,1]$ be defined as
\[
\mathfrak s(s):=\Big(\int_0^1|\dot\gamma^X_r|\,\d r\Big)^{-1}\int_0^s|\dot\gamma^X_r|\,\d r,
\] 
(if $\gamma^X$ is constant we take $\mathfrak s$ to be the identity) and let   $\bar{\mathfrak s}$ be the  map from $C([0,1],X_w)$ to itself sending  $\gamma$ to $\gamma\circ \mathfrak s$. Also, put $\sigma^t:=\bar{\mathfrak s }_\sharp\pi^{f_0(t),f_1(T(t))}$ and notice that the identity $|\dot{\gamma\circ\mathfrak s}_s|=|\mathfrak s'|(s)|\dot\gamma_{\mathfrak s}|$ and \eqref{eq:constspe} give that
\begin{equation}
\label{eq:lt}
\int |\dot\gamma_s|^2\,\d\sigma^t(\gamma)=\frac{\text{Sp}^2(t)}{(\int_0^1|\dot\gamma^X_s|\,\d s)^2}|\dot\gamma^X_s|^2,\qquad\ a.e.\ s\in[0,1].
\end{equation}
Now let $F_i:[t_i-\eps_i,t_i+\eps_i]\to[0,1]$ be the cumulative distribution function of $\nu_i$ given by $F_i(t):=\nu_i((-\infty,t])$, $i=0,1$, and consider the functions $u_i:[0,1]\to\R^+$ defined as $u_i:=\rho_i\circ F_i^{-1}$, where $\rho_i$ is the density of $\nu_i$, $i=0,1$.

Introduce also $\bar u:[0,1]\to\R^+$ as $\bar u:=\frac1{C\circ F_0^{-1}}$ (recall that $t\mapsto C(t)\geq 1$ was defined in \eqref{eq:bounddenst}), put $u:=\min\{u_0,u_1,\bar u\}$ and define $v:[0,1]\to [0,\max\{\eps_0,\eps_1\}]$ as 
\begin{equation}
\label{eq:defc}
v(t):=c\int_0^t\frac 1{u(s)}\,\d s,\qquad\text{ with }\qquad c:=\frac{\max\{\eps_0,\eps_1\}}{\int_0^1\frac 1{u(s)}\,\d s}.
\end{equation}
Then $v$ is invertible and its inverse $v^{-1}:[0,\max\{\eps_0,\eps_1\}]\to [0,1]$ is absolutely continuous and satisfies 
\begin{equation}
\label{eq:vmu}
(v^{-1})'=\frac{u\circ v^{-1}}c.
\end{equation}
Define the functions $\eta_0,\eta_1:I\to\R^+$ as
\[
\eta_i(t):=\left\{
\begin{array}{ll}
(v^{-1})'(t-t_i),&\qquad\text{ if }t\in[t_i,t_i+\max\{\eps_0,\eps_1\}],\\
0,&\qquad \text{ otherwise}
\end{array}
\right.
\]
and notice that by constructions these are probability densities.

\noindent{\bf Step 2: definition of the interpolation}. The interpolation will be built in 3 separate steps.

Let $T_0$ be the optimal transport map from $\nu_0$ to $\eta_0\mathcal L^1$, define $\hat T_0:I\times X\to C([0,1],X_w)$ as
\[
\hat T_0(t,x)_s:=\big((1-s)t+sT_0(t),x\big)
\]
and put
\[
\pi_0:=(\hat T_0)_\sharp\mu_0.
\]
Notice that we trivially have $(e_0)_\sharp\pi_0=\mu_0$. Similarly, considering the optimal map $T_1$ from $\nu_1$ to $\eta_1\mathcal L^1$ and the induced map $\hat T_{1}:I\times X\to C([0,1],X_w)$ given by $\hat T_{1}(t,x)_s:=((1-s)T_1(t)+s t,x)$, we put
\[
\pi_1:=(\hat T_1)_\sharp\mu_1,
\]
and, much like before, we have $(e_1)_\sharp\pi_1=\mu_1$. We shall now build a plan interpolating from $(e_1)_\sharp\pi_0$ to $(e_0)_\sharp\pi_1$.  
Recalling that the curve $\gamma^I$ was previously introduced, we define the map
\[
\begin{split}
\mathcal G:I\times C([0,1],X)&\quad \to\quad C([0,1],X_w),\\
(t,\gamma)&\quad\mapsto \quad s\to \mathcal G(t,\gamma)_s:=(T_0(t)+\gamma^I_s-t_0,\gamma_s).
\end{split}
\]
Then we consider the plan $\sigma\in\mathcal P(I\times C([0,1],X))$ given by
\[
\d\sigma(t,\gamma):=\d\nu_0( t)\times\d \sigma^t(\gamma),
\]
and put 
\[
\pi_{mid}:=\mathcal G_\sharp\sigma\quad \in P(C([0,1],X_w)).
\]
We claim that 
\begin{equation}
\label{eq:toint}
(e_0)_\sharp\pi_{mid}=(e_1)_\sharp\pi_0\qquad\text{ and }\qquad (e_1)_\sharp\pi_{mid}=(e_0)_\sharp\pi_1.
\end{equation}
To see the first, fix a bounded Borel function $\varphi:X_w\to\R$ and notice that
\begin{equation}
\label{eq:lcomp}
\begin{split}
\int\varphi(t,x)\,\d(e_0)_\sharp\pi_{mid}(t,x)&=\iint\varphi(\mathcal G(t,\gamma)_0)\,\d\sigma^t(\gamma)\,\d\nu_0(t)\\
&=\iint\varphi(T_0(t),\gamma_0)\,\d\sigma^t(\gamma)\,\d\nu_0(t)\\
&=\iint\varphi(T_0(t),x)\,\d(e_0)_\sharp\sigma^t(x)\,\d\nu_0(t)\\
&=\iint\varphi(T_0(t),x)\,\d\mu_{0,t}(x)\,\d\nu_0(t)\\
&=\iint\varphi(T_0(t),x)\,\d\mu_{0}(t,x)\\
&=\iint\varphi(\hat T_0(t,x)_1)\,\d\mu_{0}(t,x)=\int\varphi(t,x)\,\d(e_1)_\sharp\pi_0(t,x)
\end{split}
\end{equation}
The second in \eqref{eq:toint} follows by an analogous computation taking into account that  $t\mapsto T_1^{-1}(T_0(t)+t_1-t_0)$ is the optimal map $t\mapsto T(t)$ from $\nu_0$ to $\nu_1$ (because in 1-$d$ `optimal=monotone').

The compatibility conditions \eqref{eq:toint} and a gluing argument ensure the existence of a plan $\pi_{0mid1}\in \mathcal P(C([0,3],X_w))$ such that
\[
({\rm Res}_{[0,1]})_\sharp\pi_{0mid1}=\pi_0,\quad\text{  }\quad
({\rm Res}_{[1,2]})_\sharp\pi_{0mid1}=\pi_{mid},\quad\text{  }\quad
({\rm Res}_{[2,3]})_\sharp\pi_{0mid1}=\pi_1,
\]
where for $[a,b]\subset [0,3]$ the map ${\rm Res}_{[a,b]}:C([0,3],X_w)\to C([a,b],X_w)$ sends a curve to its restriction on $[a,b]$.

Finally, putting $\eps_{01}:=\max\{\eps_0,\eps_1\}$ we define the scaling map ${\rm Scal}:C([0,3],X_w)\to C([0,1],X_w)$ as
\[
{\rm Scal}(\gamma)_s:=\left\{
\begin{array}{ll}
\gamma_{s\eps_{01}^{-1}},&\qquad\text{ for }s\in[0,\eps_{01}],\\
\gamma_{1+\frac{s-\eps_{01}}{1-2\eps_{01}}},&\qquad\text{ for }s\in[\eps_{01},1-\eps_{01}],\\
\gamma_{2+(s+\eps_{01}-1)\eps_{01}^{-1}},&\qquad\text{ for }s\in[1-\eps_{01},1],\\
\end{array}
\right.
\]
and put
\[
\pi:={\rm Scal}_\sharp\pi_{0mid1}.
\]

\noindent{\bf Step 3: estimate of the density.} To prove that $\pi$ has bounded compression it is sufficient to show that  for every $s\in[0,1]$ it holds
\begin{equation}
\label{eq:claimdens}
(e_s)_\sharp\pi_i\leq \frac{\max\{c,1\}}{\mm_c(B_{\eps_i}(x_i))}\mm,\quad\text{for}\quad i=0,1\qquad\text{and}\qquad
(e_s)_\sharp\pi_{mid}\leq c\mm,
\end{equation}
where $c$ is the constant defined in \eqref{eq:defc}. 

We  start with the bound for $\pi_0$. Recall that $T_0$ is the optimal transport map from $\nu_0$ to $\eta_0\mathcal L^1$, thus letting $T_{0,s}(t):=(1-s)t+sT_0(t)$ and denoting by $\rho_s$ the density of $(T_{0,s})_\sharp\nu_0$, the general theory of optimal transport ensures that $s\mapsto \rho_s(T_{0,s}(t))$ is convex for $\nu_0$-a.e.\ $t$. In particular  we have
\begin{equation}
\label{eq:rhosd}
\rho_s(T_{0,s}(t))\leq\max\{\rho_0(t),\eta_0(T_0(t))\},\qquad\nu_0-a.e.\ t.
\end{equation}
Let $G_0:I\to[0,1]$ be the cumulative distribution function of $\eta_0\mathcal L^1$, i.e.\ $G_0(t):=\int_{-\infty}^t\eta_0\,\d\mathcal L^1$ and notice that by definition of $\eta_0$ we have $G(t)=v^{-1}(t-t_0)$, so that  property \eqref{eq:vmu} gives
\begin{equation}
\label{eq:nz}
\eta_0(t)=c^{-1}\,u(G_0(t)),
\end{equation}
and by the definition of $u$ we have 
\[
\eta_0(t)= c^{-1} \, u(G_0(t))\leq c^{-1}\,u_0(G_0(t))=c^{-1}\,\rho_0(F_0^{-1}(G_0(t))),\qquad\eta_0\mathcal L^1-a.e.\ t.
\]
Since the optimal map $T_0$ is the inverse of $F_0^{-1}\circ G_0$, from \eqref{eq:rhosd} we get that 
\[
\rho_s(T_{0,s}(t))\leq\max\{ c^{-1},1\}\,\rho_0(t),\qquad\rho_0\mathcal L^1-a.e.\ t,
\]
for every $s\in[0,1]$. Now notice that with computations similar to those in \eqref{eq:lcomp} we see that
\[
\d(e_s)_\sharp\pi(t,x)=\rho_s(t)\d t\times\d\mu_{0,T_{0,s}^{-1}(t)}(x),
\] 
in particular $(e_s)_\sharp\pi\ll\mm_c$ and for its density we have
\[
\begin{split}
\frac{\d(e_s)_\sharp\pi_0}{\d\mm_c}((T_{0,s}(t),x)&=\rho_s(T_{0,s}(t))\frac{\d\mu_{0,t}}{\d\mm}(x)\\
& \leq \max\{ c^{-1},1\}\rho_0(t)\frac{\d\mu_{0,t}}{\d\mm}(x)
=\max\{ c^{-1},1\}\frac{\d\mu_0}{\d\mm_c}(t,x)\leq \frac{\max\{ c^{-1},1\}}{\mm_c(B_{\eps_0}(x_0))},
\end{split}
\]
for every $s\in[0,1]$ so that in this case the claim \eqref{eq:claimdens} is proved. The bound for $\pi_1$ is obtained in the same way, so we turn to the one on $\pi_{mid}$.

To this aim, start noticing that from the identity 
\[
\begin{split}
\int\varphi(t,s)\,\d(e_s)_\sharp\pi_{mid}&=\int\left(\int\varphi(T_0(t)+\gamma^I_s-t_0,\gamma_s)\,\d\sigma^t(\gamma)\right)\,\d\nu_0(t)\\
&=\int\left(\int\varphi(T_0(t)+\gamma^I_s-t_0,\gamma_s)\,\d\pi^{f_0(t),f_1(T(t))}(\gamma)\right)\,\d\nu_0(t)\\
&=\int\left(\int\varphi(t+\gamma^I_s-t_0,x)\,\d(e_s)_\sharp\pi^{f_0(T_0^{-1}(t)),f_1(T(T_0^{-1}(t)))}(x)\right)\eta_0(t)\,\d t
\end{split}
\]
valid for any bounded Borel function $\varphi$, we see that $(e_s)_\sharp\pi_{mid}\ll\mm_c$ and
\[
\frac{\d(e_s)_\sharp\pi_{mid}}{\d\mm_c}(t-t_0+\gamma^I_s,x)=\eta_0(t)\frac{\d(e_s)_\sharp\pi^{f_0(T_0^{-1}(t)),f_1(T(T_0^{-1}(t)))}}{\d\mm}(x),\qquad\eta_0\mathcal L^1\times\mm-a.e.\ (t,x).
\]
Recalling the bound \eqref{eq:bounddenst} we therefore obtain
\[
\frac{\d(e_s)_\sharp\pi_{mid}}{\d\mm_c}(t-t_0+\gamma^I_s,x)\leq \eta_0(t)C(T_0^{-1}(t)),\qquad\eta_0\mathcal L^1\times\mm-a.e.\ (t,x).
\]
Now notice that  \eqref{eq:nz} and the definition of $u$ give
\[
\eta_0(t)\leq \frac{c^{-1}}{C(F_0^{-1}(G_0(t)))},
\]
and since $T_0^{-1}=F_0^{-1}\circ G_0$, these last two bounds give our claim.

\noindent{\bf Step 4: estimate of the kinetic energy.}  Notice that by the very definition of $\pi$ we have
\begin{equation}
\label{eq:ktot}
\begin{split}
\iint_0^1|\dot\gamma_s|^2\,\d s\,\d\pi(\gamma)&=\frac{1}{\eps_{01}}\iint_0^1|\dot\gamma_s|^2\,\d s\,\d\pi_0(\gamma)+\frac1{\eps_{01}}\iint_0^1|\dot\gamma_s|^2\,\d s\,\d\pi_1(\gamma)\\
&\qquad\qquad+\frac{1}{1-2\eps_{01}}\iint_0^1|\dot\gamma_s|^2\,\d s\,\d\pi_{mid}(\gamma)
\end{split}
\end{equation}

We start estimating the energy of $\pi_0$. Notice that since the $\d_w$-metric speed of the curve $s\mapsto \hat T_0(t,x)_s$ is constantly equal to $|T_0(t)-t|$, we have
\[
\begin{split}
\iint|\dot\gamma_s|^2\,\d s\,\d\pi_0(\gamma)=\int|T_0(t)-t|^2\,\d \mu_0(t,x)=\int|T_0(t)-t|^2\,\d \nu_0(t).
\end{split}
\]
Now observe that since ${\rm spt}(\nu_0)\subset [t_0-\eps_0,t_0+\eps_0]$ and ${\rm spt}(\eta_0)\subset [t_0,t_0+\eps_{01}]$, in transporting $\nu_0$ to $\eta_0\mathcal L^1$ no point is moved for more than $\eps_0+\eps_{01}$, therefore we have
\begin{equation}
\label{eq:enpi0}
\iint|\dot\gamma_s|^2\,\d s\,\d\pi_0(\gamma)\leq |\eps_0+\eps_{01}|^2\leq 4\eps_{01}^2.
\end{equation}
An analogous argument yields the bound
\begin{equation}
\label{eq:enpi1}
\iint|\dot\gamma_s|^2\,\d s\,\d\pi_1(\gamma)\leq 4\eps_{01}^2.
\end{equation}
We pass to the energy of $\pi_{mid}$.  By the very definition of $\d_w$, the $\d_w$-squared speed of $s\mapsto (\gamma^I_{s}+T_0(t)-t_0,\gamma_{s})$ is equal to $|\dot\gamma^I_s|^2+w_\d^2(\gamma^I_{s}+T_0(t)-t_0)|\dot\gamma_s|^2$ and thus we have
\[
\begin{split}
\iint_0^1|\dot\gamma_s|^2\,\d s\,\d\pi_{mid}(\gamma)&=\iiint_0^1 |\dot\gamma^I_s|^2+w_\d^2(\gamma^I_{s}+T_0(t)-t_0)|\dot\gamma_s|^2\,\d s\,\d \sigma^t(\gamma) \,\d\nu_0( t)\\
&=\int_0^1 |\dot\gamma^I_s|^2\,\d s+\iint_0^1 w_\d^2(\gamma^I_{s}+T_0(t)-t_0)\left(\int|\dot\gamma_s|^2\,\d \sigma^t(\gamma) \right)\,\d s\,\d\nu_0( t)\\
\text{by \eqref{eq:lt}}\qquad&=\int_0^1 |\dot\gamma^I_s|^2\,\d s+\iint_0^1 w_\d^2(\gamma^I_{s}+T_0(t)-t_0)|\dot\gamma^X_s|^2\frac{\text{Sp}^2(t)}{(\int_0^1|\dot\gamma^X_s|\,\d s)^2}\,\d s\,\d\nu_0( t).
\end{split}
\]
Now observe that  Lemma \ref{le:balls} ensures that $\int_0^1|\dot\gamma^X_s|\,\d s=\d(x_0,x_1)$, hence from \eqref{eq:sppi} we obtain
\[
\frac{1}{(\int_0^1|\dot\gamma^X_s|\,\d s)^2}\lims_{\eps_0,\eps_1\downarrow0}\sup_{t\in[t_0-\eps_0,t_0+\eps_0]}{\rm Sp}^2(t)=1
\]
and from this information,  the continuity of $w_\d$, the weak convergence of $\nu_0$ to $\delta_{t_0}$ as $\eps_0\downarrow0$ we obtain
\[
\begin{split}
\lims_{\eps_0,\eps_1\downarrow0}\iint_0^1|\dot\gamma_s|^2\,\d s\,\d\pi_{mid}(\gamma)&\leq \int_0^1 |\dot\gamma^I_s|^2+w^2_\d(\gamma^I_s)|\dot\gamma^X_s|^2\,\d s=\d_w^2\big((t_0,x_0),(t_1,x_1)\big),
\end{split}
\]
which together with \eqref{eq:enpi0}, \eqref{eq:enpi1} and \eqref{eq:ktot} gives the conclusion.

\noindent{\bf Step 5: conclusion.} We assumed initially the existence of a geodesic $\gamma=(\gamma^I,\gamma^X)$ with $\gamma^I$ having image in the interior of $I$. To remove such assumption, it is sufficient to observe that there always exists a sequence of curves $\gamma_n=(\gamma^I_n,\gamma^X_n)$ with the same boundary data whose length converges to $\d_w((t_0,x_0),(t_1,x_1))$ and  such that $\gamma^I_n$ has image in the interior of $I$ for every $n$. Then we can repeat the above arguments with $\gamma_n$ in place of $\gamma$ and conclude by diagonalization.
\end{proof}

Summing up what we proved so far we obtain:
\begin{theorem}
Let $(X,\d,\mm)$ be an a.e.\ locally doubling and measured-length space with finite mass, $I\subset \R$ a closed, possibly unbounded, interval and $w_\d,w_\mm:I\to[0,+\infty)$ continuous  functions.  Assume that $w_\mm$ is strictly positive in the interior of $I$.

Then the warped product space $(X_w,\d_w,\mm_w)$ is a.e.\ doubling and measured-length. In particular, it has the Sobolev-to-Lipschitz property.
\end{theorem}
\begin{proof}
Direct consequence of Propositions \ref{prop:dw}, \ref{prop:mlw} and \ref{prop:stl}
\end{proof}

\def\cprime{$'$} \def\cprime{$'$}

\end{document}